\documentclass[journal]{IEEEtran}

\usepackage{comment}
\usepackage{url}

\usepackage{graphicx}
\usepackage{amssymb}
\usepackage{amsmath,amsfonts,amssymb,euscript,epsfig,psfrag,
enumerate,float,afterpage, subfigure}%

\newtheorem{thm}{Theorem}

\newtheorem{lem}{Lemma}
\newcommand{\expect}[1]{\mathbb{E}\left\{#1\right\}}
\newcommand{\defequiv}{\mbox{\raisebox{-.3ex}{$\overset{\vartriangle}{=}$}}}

\usepackage{cite}

\begin{document}

%\title{Dynamic Cooperation in Cognitive Femtocell Networks}
\title{Opportunistic Cooperation in Cognitive Femtocell Networks}

\author{Rahul~Urgaonkar, Michael~J.~Neely \\
 \IEEEcompsocitemizethanks{
\IEEEcompsocthanksitem Rahul Urgaonkar and Michael J. Neely are with the Department
of Electrical Engineering, University of Southern California, Los Angeles, CA
90089. E-mail: urgaonka@usc.edu, mjneely@usc.edu, Web: http://www-scf.usc.edu/$\sim$urgaonka
\IEEEcompsocthanksitem  This material is supported in part by one or more of the following: the DARPA IT-MANET program grant W911NF-07-0028, the NSF
Career grant CCF-0747525, and continuing through participation in the Network Science Collaborative Technology Alliance sponsored by
the U.S. Army Research Laboratory.}
}

\maketitle
\begin{abstract}
We investigate opportunistic cooperation between unlicensed secondary users and
legacy primary users in a cognitive radio network. Specifically, we consider
a model of a cognitive network where a secondary user can cooperatively
transmit with the primary user in order to improve the latter's effective transmission rate.
In return, the secondary user gets more opportunities for transmitting its own data when
the primary user is idle. This kind of interaction between the primary and secondary users
is different from the traditional \emph{dynamic spectrum access} model in which
the secondary users try to avoid interfering with the primary users while 
seeking transmission opportunities on vacant primary channels. 
In our model, the secondary users need to balance the desire to cooperate
more (to create more transmission opportunities) with the need for maintaining sufficient energy levels for their
own transmissions. Such a model is applicable in the emerging area of cognitive femtocell networks.
We formulate the problem of maximizing the secondary user throughput
subject to a time average power constraint under these settings. This is a constrained
Markov Decision Problem and conventional solution techniques based on dynamic programming require either 
extensive knowledge of the system dynamics or learning based approaches that suffer from large convergence times.
However, using the technique of Lyapunov optimization, we design a novel \emph{greedy} and
\emph{online} control algorithm that overcomes these challenges and is provably optimal.
\end{abstract}


\section{Introduction}
\label{section:femto_intro}

Much prior work on resource allocation in
cognitive radio networks has focused on the \emph{dynamic spectrum access} model
\cite{cognitive-survey, zhao-survey}
in which  the secondary users seek transmission opportunities for their packets on vacant primary channels in frequency, 
time, or space. 
Under this model, the primary users are assumed to be oblivious of the presence of the secondary users and 
transmit whenever they have data to send. 
Secondly, a collision model is assumed
for the physical layer in which if a secondary user transmits on a busy primary channel, then there is a 
collision and both packets are lost. 
We considered a similar model in our prior work
\cite{CNC} where the objective was 
to design an opportunistic scheduling policy for the secondary users that maximizes their throughput utility while
providing tight reliability guarantees on the maximum number of collisions suffered by a primary user over \emph{any} 
given time interval. We note that this formulation does not consider the possibility of any cooperation between the primary and secondary users. 
Further, it assumes that the secondary user activity does not affect the primary user channel occupancy process.

There is a growing body of work that investigates alternate models for the interaction between the primary and secondary 
users in a cognitive radio network. In particular, the idea of cooperation at the physical layer has been considered from an
information-theoretic perspective in many works (see \cite{goldsmith} and the references therein). 
These are motivated by the work on the classical interference and relay channels 
\cite{carleial, han, cover_gamal, cover}.
%\cite{cover}.
  %\emph{Cooperation at the MAC [ephremides] and network layer[simeone]...[Devroye][tarokh] papers.
  %Refs on cogintion, cooperation/collaboration and competition [][].
 %The main assumption here is that the secondary user somehow already knows the primary message/codebook, etc.
 %Also, game theoretic approach. Work on spectrum leasing [], also [ubli spectrum shaping].}
The main idea in these works is that the resources of the secondary user can be utilized to improve the performance of the primary 
transmissions. In return, the secondary user can obtain more transmission opportunities
for its own data when the primary channel is idle.

These works mainly treat the problem from a physical layer/information-theoretic perspective and 
do not consider upper layer issues such as queueing dynamics, higher priority for primary user, etc.
Recent work that addresses some of these issues includes \cite{SBNS07, simeone, ZZ09, kriki, brong}. 
Specifically, \cite{SBNS07} considers the scenario where the secondary user acts as a relay for those
packets of the primary user that it receives successfully but which are not received by the primary destination.
It derives the stable throughput of the secondary user under this model. \cite{simeone, ZZ09} use a
Stackelberg game framework to study spectrum leasing strategies in \emph{cooperative cognitive radio networks} where
the primary users lease a portion of their licensed spectrum to secondary users in return for cooperative relaying.
\cite{kriki, brong} study and compare different physical layer strategies for relaying in such
cognitive cooperative systems. 
 %\emph{Add details on these works...Mention cooperative cognitive radio networks.. Also diff. models for cognitive radios?} 
An important consequence of this interaction between the primary and secondary users 
is that the secondary user activity can now potentially 
influence the primary user channel occupancy process. However, there has been little work in studying this
scenario. Exceptions include the work in \cite{marco}
that considers a two-user setting where collisions caused by the opportunistic transmissions of the secondary user result in 
retransmissions by the primary user.

%[ucdavis]....

In this paper, we study the problem of opportunistic cooperation in cognitive networks from a \emph{network utility maximization} perspective,
specifically taking into account the above mentioned higher-layer aspects. 
To motivate the problem and illustrate the design issues involved, we first consider a simple network consisting of 
one primary and one secondary user and their respective access points in Sec. \ref{section:femto_basic}. 
This can model a practical scenario of recent interest, namely a cognitive femtocell \cite{gur, JL10}, as
discussed in Sec. \ref{section:femto_basic}. We assume that the secondary user can cooperatively transmit
with the primary user to increase its transmission success probability. 
In return, the secondary user can get more opportunities for transmitting its own data when the
primary user is idle. We formulate the problem of maximizing the secondary user throughput subject
to time average power constraints in Sec. \ref{section:femto_objective}. 
 %Here, a secondary user operates over a much smaller area (a femtocell) while the primary
 %A key difference between this problem setting and most of the prior work on resource allocation in cognitive radio networks
 %is that here, the evolution of the system state depends on the control actions taken by the secondary user. 

Unlike most of the prior work on resource allocation in cognitive radio networks,
the evolution of the system state for this problem depends on the control actions taken by the secondary user. 
Here, the system state refers to the channel occupancy state of the primary user.
Because of this dependence, this problem becomes a constrained Markov Decision Problem (MDP) and
the greedy ``drift-plus-penalty'' minimization technique of Lyapunov optimization \cite{neely-NOW}
that we used in \cite{CNC} is no longer optimal.
Such problems are typically tackled using Markov Decision Theory and Dynamic Programming \cite{altman, bertsekas}. 
For example, \cite{marco} uses these
tools to derive structural results on optimal channel access strategies
in a similar two-user setting where collisions caused by the opportunistic transmissions of the secondary user cause the
primary user to retransmit its packets. However, this approach requires either extensive knowledge of the dynamics of
the underlying network state (such as state transition probabilities) or learning based approaches that suffer from large convergence times.

%ur motivation comes from the
%observation that secondary users can use cooperative techniques to improve the reliability of the 
%primary transmissions. This can in turn
%create more transmission opportunities for themselves...

Instead, in Sec. \ref{section:solution}, we use the recently developed framework of maximizing the \emph{ratio} of the expected total reward over 
the expected length of a renewal frame \cite{chih-ping, asilomar10, neely_new} to design a control algorithm. 
This framework extends the classical Lyapunov optimization method \cite{neely-NOW} to tackle a 
more general class of MDP problems where the system evolves over renewals and where the length of a renewal frame can be
affected by the control decisions during that period.
The resulting solution has the following structure: Rather than minimizing a ``drift-plus-penalty'' term every slot, it minimizes a 
``drift-plus-penalty ratio'' over each renewal frame. This can be achieved by solving a sequence of unconstrained 
\emph{stochastic shortest path} (SSP) problems and implementing the solution over every renewal frame.

While solving such SSP problems can be simpler than the original constrained MDP, it may still require knowledge of
the dynamics of the underlying network state. Learning based techniques for solving such problems 
by sampling from the past observations have been considered in \cite{neely_mdp}.
However, these may suffer from large convergence times. Remarkably, in Sec. \ref{section:solving},
we show that for our problem, the ``drift-plus-penalty ratio'' method results in an online control 
algorithm that \emph{does not require any knowledge of the network dynamics or explicit learning}, yet is optimal. 
In this respect, it is similar to the traditional greedy ``drift-plus-penalty''
minimizing algorithms of \cite{neely-NOW}. We then extend the basic model to incorporate multiple secondary users as well as time-varying channels in
Sec. \ref{section:femto_extensions}. Finally, we present simulation results in Sec. \ref{section:femto_sim}.

%The rest of the paper is organized as follows. We first consider a basic model with one secondary and one primary user and their
%respective base stations in Sec. \ref{section:basic}. In Sec. 

\section{Basic Model}
\label{section:femto_basic}

%\begin{figure}
%\centering
%\includegraphics[width=5cm]{3node}
%\caption{3 node network under consideration}\label{fig:3node}
%\end{figure}

\begin{figure}
\centering
\includegraphics[width=7cm, height=4cm]{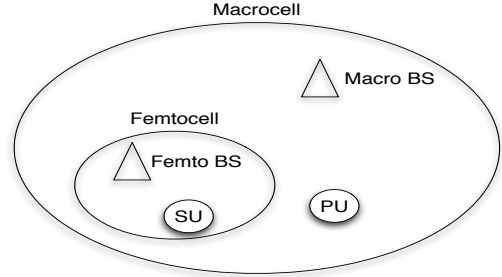}
\caption{Example femtocell network with primary and secondary users.}
\label{fig:femto}
\end{figure}

We consider a network with one primary user (PU), one secondary user (SU) and their respective base stations (BS). 
The primary user is the licensed owner of the channel while the secondary user tries to send its own
data opportunistically when the channel is not being used by the
primary user. This model can capture a femtocell scenario where the primary user is a legacy mobile user
that communicates with the macro base station over licensed spectrum (Fig. \ref{fig:femto}).
  %\emph{Could make a better figure. See recent infocom refs}. 
The secondary user is the femtocell user that
does not have any licensed spectrum of its own and tries to send data opportunistically to the femtocell base station
over any vacant licensed spectrum. 
Similar models of \emph{cooperative cognitive radio networks} have been considered in \cite{SBNS07, simeone, ZZ09, kriki, brong}. 
 %Similar models have been considered in [][]. 
This can also model a single server queueing system with two classes of arrivals where one class has a strictly higher priority
over the other class.

\begin{comment}
We consider a time-slotted model. We assume that the system operates over \emph{renewal frames}. Specifically, the timeline can be divided into
successive non-overlapping frames of duration $T[k]$ slots where $k \in \{1, 2, 3, \ldots \}$ represents the frame number (see Fig. \ref{fig:renewal}). 
The start time of frame $k$ is denoted by $t_k$ with $t_1 = 0$ and {is a time when the system state is ``refreshed'', to be made precise below.}
The length of frame $k$ is given by $T[k] \defequiv t[k+1] - t[k]$. 
For each $k$, the frame length $T[k]$ is assumed to be a random function of 
the control decisions taken during that frame and is independent of the control actions from previous frames.
Each renewal frame can be further divided into two periods: PU Idle and PU Busy. 
The ``PU Idle'' period corresponds to the slots when the primary user does not have any
packet to send to its base station and is idle.
The ``PU Busy'' period corresponds to the slots when the primary user is transmitting
its packets to its base station over the licensed spectrum. 
As shown in Fig. \ref{fig:renewal}, 
every renewal frame starts with the ``PU Idle'' period which is followed by the ``PU Busy''
period and ends when the primary user becomes idle again. 
\end{comment}

We consider a time-slotted model. We assume that the system operates over a frame-based structure.
Specifically, the timeline can be divided into
successive non-overlapping frames of duration $T[k]$ slots where $k \in \{1, 2, 3, \ldots \}$ represents the frame number (see Fig. \ref{fig:renewal}). 
The start time of frame $k$ is denoted by $t_k$ with $t_1 = 0$.
The length of frame $k$ is given by $T[k] \defequiv t_{k+1} - t_k$. 
For each $k$, the frame length $T[k]$ is a random function of 
the control decisions taken during that frame.
Each frame can be further divided into two periods: PU Idle and PU Busy. 
The ``PU Idle'' period corresponds to the slots when the primary user does not have any
packet to send to its base station and is idle.
The ``PU Busy'' period corresponds to the slots when the primary user is transmitting
its packets to its base station over the licensed spectrum. 
As shown in Fig. \ref{fig:renewal}, 
every frame starts with the ``PU Idle'' period which is followed by the ``PU Busy''
period and ends when the primary user becomes idle again. 
In the basic model, we assume that the primary user receives new packets every slot according to an i.i.d. Bernoulli arrival 
process $A_{pu}(t)$ with rate $\lambda_{pu}$ packets/slot. 
This means that the length of the ``PU Idle'' period of any frame 
is a geometric random variable with parameter $\lambda_{pu}$. However, the
length of the ``PU Busy'' period depends on the secondary user control decisions as
discussed below.

In any slot $t$, if the primary user has a non-zero queue backlog, it transmits one packet
to its base station. We assume that the transmission of each packet takes one slot.
If the transmission is successful, the packet is removed from the primary user queue. 
However, if the transmission fails, the packet is retained in the queue for future retransmissions.
 %\emph{No limit on number of retransmissions/infinite buffer size..}
The secondary user cannot transmit its packets when the channel is being used by the primary user. It can
transmit its packets only during the ``PU Idle'' period of the frame and must stop its transmission whenever the primary user becomes active again.
However, the secondary user can transmit cooperatively with the primary user in the ``PU Busy'' period to increase its 
transmission success probability.
This has the effect of decreasing the expected length of the ``PU Busy'' period.
In order to cooperate,  the secondary user must allocate its power resources to help relay the primary user packet.
This cooperation can take place in several ways depending on the cooperative protocol being used (see \cite{kriki} for some examples). 
In this simple model, these details are captured by the resulting probability of successful transmission.

\begin{figure}
\centering
\includegraphics[width=8.5cm]{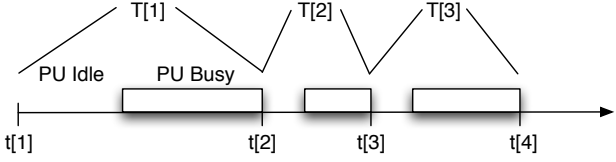}
%\caption{Renewal structure of the problem under consideration. Each renewal frame consists of two periods: PU Idle and PU Busy.}
\caption{Frame-based structure of the problem under consideration. Each frame consists of two periods: PU Idle and PU Busy.}
\label{fig:renewal}
\end{figure}

The reason why the secondary user may want to cooperate is because this can potentially increase the
number of time slots in the future in which the primary user does not have any
data to send as compared to a non-cooperative strategy. This can 
create more opportunities for the secondary user to transmit its own packets. 
However, note that the trivial strategy of cooperating whenever possible may lead to a scenario where the
secondary user does not have enough power for its own data transmission. 
Thus, the secondary user needs to decide whether it should cooperate or not considering these two
opposing factors.

The probability of a successful primary transmission 
depends on the control actions such as power allocation and cooperative transmission
decisions by the secondary user. This is discussed in detail in the next section.
In this model, we assume that the network controller cannot control the primary user
actions. However, it can control the secondary user decisions on cooperation and the associated 
power allocation.

\subsection{Control Decisions and Queueing Dynamics}
\label{section:femto_decisions}

Let $Q_{pu}(t), Q_{su}(t) \in \{ 0, 1, 2, \ldots \}$ represent the primary and 
secondary user queues respectively in slot $t$. 
New packets arrive at the secondary user according to an i.i.d. process 
$A_{su}(t)$ of rate $\lambda_{su}$ packets/slot respectively. We assume that there exists a finite
constant $A_{max}$ such that $A_{su}(t) \leq A_{max}$ for all $t$.
Every slot, an admission control decision determines $R_{su}(t)$, the number of new packets to admit into
the secondary user queue. Further, every slot, depending on whether the primary user is busy or idle,
resource allocation decisions are made as follows.
When $Q_{pu}(t) > 0$, this represents the secondary user decision on cooperative transmission and the 
corresponding power allocation $P_{su}(t)$. When $Q_{pu}(t) = 0$, this corresponds to the secondary user decision 
on its own transmission and the corresponding power allocation $P_{su}(t)$.
We assume that in each slot, the secondary user can choose its power allocation $P_{su}(t)$ from a set $\mathcal{P}$ of
possible options. Further, this power allocation is subject to a long-term average power constraint $P_{avg}$ and
an instantaneous peak power constraint $P_{max}$. 
For example, $\mathcal{P}$ may contain only two options $\{0, P_{max}\}$ which represents
``Remain Idle'' and ``Cooperate/Transmit at Full Power''. As another example, $\mathcal{P} = [0, P_{max}]$ such that
$P_{su}(t)$ can take any value between $0$ and $P_{max}$.

 %\emph{Let $\mathcal{I}(t)$ denote the collective control action in slot $t$.} 

Suppose the primary user is active in slot $t$ and the secondary user allocates power $P(t)$ for cooperative transmission. 
Then the random success/failure outcome of the primary transmission is given by an indicator variable $\mu_{pu}(P(t))$ and the 
success probability is given by $\phi(P(t)) = \expect{\mu_{pu}(P(t))}$. The function $\phi(P)$ is known to the network controller
and is assumed to be non-decreasing in $P$.
 %\footnote{This abstract function captures the essential features of a general cooperative transmission scheme.}. 
However, the value of the random outcome $\mu_{pu}(P(t))$ may not be known beforehand. Note that setting $P(t) = 0$
corresponds to a non-cooperative transmission and the success probability for this case becomes $\phi(0)$ and we denote this by 
$\phi_{nc}$. Likewise, we denote $\phi(P_{max})$ by $\phi_c$. Thus, $\phi_{nc} \leq \phi(P(t)) \leq \phi_c$ for all $P(t) \in \mathcal{P}$.

We assume that $\lambda_{pu}$ is such that it can be supported even when the secondary user never cooperates, i.e., 
$\lambda_{pu} < \phi_{nc}$. This means that the primary user queue is stable even if there is no cooperation.
 %expected length of .... is always finite under any .....
Further, for all $k$, the frame length $T[k] \geq 1$ and there exist finite constants $T_{min}, T_{max}$ such that
under all control policies, we have:
\begin{align*}
1 \leq T_{min} \leq \expect{T[k]} \leq T_{max} 
\end{align*}
Specifically, $T_{min}$ can be chosen to be the expected frame length when the secondary user always cooperates with 
full power while $T_{max}$ can be chosen to be the expected frame length when the secondary user never cooperates.
Using  Little's Theorem, we have that: 
\begin{align*}
\frac{T_{min}}{T_{min} + 1/\lambda_{pu}} = \frac{\lambda_{pu}}{\phi_{c}}
\end{align*}
Similarly, we have: 
\begin{align*}
\frac{T_{max}}{T_{max} + 1/\lambda_{pu}} = \frac{\lambda_{pu}}{\phi_{nc}}
\end{align*}
Using these, we have:
 %these can be calculated as follows:
\begin{align}
T_{min} \defequiv \frac{\phi_c}{(\phi_{c} - \lambda_{pu})\lambda_{pu}}, \;
T_{max} \defequiv \frac{\phi_{nc}}{(\phi_{nc} - \lambda_{pu})\lambda_{pu}} 
\label{eq:t_min_max}
\end{align}
Finally, there exists a finite constant $D$ such that the expectation of the second moment of a frame size, 
$\expect{T^2[k]}$, satisfies the following for all $k$, regardless of the policy:
\begin{align}
\expect{T^2[k]} \leq D
\label{eq:second_moment}
\end{align}
This follows from the assumption that the primary user queue is stable even if there is no cooperation.
In Appendix C, we exactly compute such a $D$ that satisfies (\ref{eq:second_moment}).

When the primary user is idle in slot $t$ and the secondary user allocates power $P(t)$ for its own transmission, 
it gets a service rate given by $\mu_{su}(P(t))$. 
 %\emph{We assume that there exists a finite constant $\beta > 0$ such that $\mu_{su}(P) \leq \beta P$ for all $P \in \mathcal{P}$...} 
This can represent the success probability of a secondary transmission with a Bernoulli service process. This 
can also be used to model more general service processes.
We assume that there exists a finite constant $\mu_{max}$ such that
$\mu_{su}(P) \leq \mu_{max}$ for all $P \in \mathcal{P}$.
 %For notational convenience, in the rest of the paper, we will denote $\mu_{su}(P(t))$ by $\mu_{su}(t)$ noting that
 %the dependence on $P(t)$ is implicit.
 %\emph{the success/failure outcome  is given by an indicator random variable $\mu_{su}(P(t))$. Is this needed?}

Given these control decisions, the primary and secondary user queues evolve as follows:
\begin{align}
Q_{pu}(t+1) = \max[Q_{pu}(t) - \mu_{pu}(P(t)), 0] + A_{pu}(t)
\label{eq:ch4_queues1} \\
Q_{su}(t+1) = \max[Q_{su}(t) - \mu_{su}(P(t)), 0] + R_{su}(t)
\label{eq:ch4_queues2}
\end{align}
where $R_{su}(t) \leq A_{su}(t)$.

\begin{comment}

\subsection{Renewal Assumptions}
\label{section:renewal}

For each $k$, the frame length $T[k]$ is assumed to be a random function of 
the control decisions taken during that frame and is independent of the control actions from previous frames.
Further, $T[k] \geq 1$ for all $k$ and there exist finite constants $T_{min}, T_{max}$ such that
under all control policies, we have:
\begin{align*}
1 \leq T_{min} \leq \expect{T[k]} \leq T_{max} 
\end{align*}
Specifically, $T_{min}$ can be chosen to be the expected frame length when the secondary user always cooperates with 
full power while $T_{max}$ can be chosen to be the expected frame length when the secondary user never cooperates.
Using  Little's Theorem, these can be calculated as follows:
\begin{align}
T_{min} \defequiv \frac{\phi_c}{(\phi_{c} - \lambda_{pu})\lambda_{pu}}, \qquad T_{max} \defequiv \frac{\phi_{nc}}{(\phi_{nc} - \lambda_{pu})\lambda_{pu}} 
\label{eq:t_min_max}
\end{align}
Finally, there exists a finite constant $D$ such that the expectation of the second moment of a frame size, 
$\expect{T^2[k]}$, satisfies the following for all $k$, regardless of the policy:
\begin{align}
\expect{T^2[k]} \leq D
\label{eq:second_moment}
\end{align}
This follows from the assumption that the primary user queue is stable even if there is no cooperation.
In Appendix C, we exactly compute such a $D$ that satisfies (\ref{eq:second_moment}).

\end{comment}

\subsection{Control Objective}
\label{section:femto_objective}

Consider any control algorithm that makes admission control decision $R_{su}(t)$ and power allocation $P(t)$ every slot 
subject to the constraints described in
Sec. \ref{section:femto_decisions}. Note that if the primary queue backlog $Q_{pu}(t) > 0$, then this power is used for cooperative
transmission with the primary user. If $Q_{pu}(t) = 0$, then this power is used for the secondary user's
own transmission. Define the following time-averages under this algorithm:
\begin{align*}
& \overline{R}_{su} \defequiv \lim_{t\rightarrow\infty} \frac{1}{t} \sum_{\tau = 0}^{t-1} \expect{R_{su}(\tau)} \\
& \overline{P}_{su} \defequiv \lim_{t\rightarrow\infty} \frac{1}{t} \sum_{\tau = 0}^{t-1} \expect{P(\tau)} \\
& \overline{\mu}_{su} \defequiv \lim_{t\rightarrow\infty} \frac{1}{t} \sum_{\tau = 0}^{t-1} \expect{\mu_{su}(P(\tau))}
%\label{eq:mu_su_avg}
\end{align*}
where the expectations above are with respect to the potential randomness of the control algorithm.
Assuming for the time being that these limits exist,
our goal is to design a joint admission control and power allocation
policy that maximizes the throughput of the secondary user subject
to its average and peak power constraints and the scheduling constraints imposed by 
the basic model. 
Formally, this can be stated as a stochastic optimization problem as follows:
\begin{align}
\textrm{Maximize:} \;\; & \overline{R}_{su} \nonumber \\ 
\textrm{Subject to:} \;\; & 0 \leq R_{su}(t) \leq A_{su}(t) \; \forall t \nonumber \\
& P(t) \in \mathcal{P} \; \forall t \nonumber \\
& \overline{R}_{su} \leq \overline{\mu}_{su} \nonumber\\
& \overline{P}_{su} \leq P_{avg}
\label{eq:cmdp}
\end{align}
It will be useful to define the primary queue backlog $Q_{pu}(t)$ as the ``state'' for this control problem.
This is because the state of this queue (being zero or nonzero) affects the control options as described before. 
Note that the control decisions on cooperation affect the dynamics of this queue. Therefore, 
problem (\ref{eq:cmdp}) is an instance of a constrained Markov decision problem \cite{altman}.
 It is well known that in order to obtain an optimal control policy, it is sufficient to consider only the class of stationary, randomized
policies that take control actions only as a function of the current system state (and independent of past history).
 %It can be shown that this class of policies is optimal among all possible control policies. 
 %Given this, our network control problem can be formulated as a Constrained Markov Decision Problem \cite{altman}. 
A general control policy in this class is characterized by a stationary probability distribution over the control
action set for each system state. 
%Let $\upsilon^*$ denote the optimal value of the objective in (\ref{eq:cmdp}). Then we
%have the following:
Let $\upsilon^*$ denote the optimal value of the objective in (\ref{eq:cmdp}). Then using
standard results on constrained Markov Decision problems 
\cite{altman, puterman, bertsekas2}, we have the following:

\begin{lem} (Optimal Stationary, Randomized Policy): 
%If the  workload process $W(t)$ and auxiliary process $S(t)$ are i.i.d. over slots, then 
There exists a stationary, randomized policy \emph{STAT} that takes control decisions $R_{su}^{stat}(t), P_{su}^{stat}(t)$
every slot purely as a (possibly randomized) function of the current state $Q_{pu}(t)$ while satisfying the
constraints $R_{su}^{stat}(t) \leq A_{su}(t), P_{su}^{stat}(t) \in \mathcal{P}$ for all $t$ and
provides the following guarantees:
\begin{align}
& \overline{R}_{su}^{stat} = \upsilon^* \label{eq:femto_stat1} \\
& \overline{R}_{su}^{stat} \leq \overline{\mu}_{su}^{stat} \label{eq:femto_stat2}\\
& \overline{P}_{su}^{stat} \leq P_{avg} \label{eq:femto_stat3}
\end{align}
where $\overline{R}_{su}^{stat}, \overline{\mu}_{su}^{stat}, \overline{P}_{su}^{stat}$ denote the
time-averages under this policy.
\label{lem:femto_one}
\end{lem}
%\begin{proof} 
%\end{proof}

We note that the conventional techniques to solve (\ref{eq:cmdp}) that are based
on dynamic programming \cite{bertsekas} require either extensive knowledge of
the system dynamics or learning based approaches that suffer
from large convergence times. Motivated by the
recently developed extension to the technique of Lyapunov optimization in \cite{chih-ping, asilomar10, neely_new}, 
we take an different approach to this problem in the next section.

\section{Solution Using The ``Drift-plus-Penalty'' Ratio Method}
\label{section:solution}

Recall that the start of the $k^{th}$ frame, $t_k$, is defined as the first slot when the primary user becomes idle 
after the ``PU Busy'' period of the $(k-1)^{th}$ frame. 
Let $Q_{su}(t_k)$ denote the secondary user queue backlog at time $t_k$.
 %Let $U_{su}(t(k))$ denote the secondary user queue backlog at time $t(k)$.
 %the start of the $k^{th}$ renewal frame. 
Also let $P(t)$ be the power expenditure incurred by the
secondary user in slot $t$. 
For notational convenience, in the following we will denote $\mu_{su}(P(t))$ by $\mu_{su}(t)$ noting
the dependence on $P(t)$ is implicit.
Then the queueing dynamics of $Q_{su}(t_k)$ satisfies the following:
\begin{align}
Q_{su}(t_{k+1}) &\leq \max[Q_{su}(t_k) - \sum_{t=t_k}^{t_{k+1}-1} \mu_{su}(t), 0] \nonumber\\
&\qquad + \sum_{t=t_k}^{t_{k+1}-1} R_{su}(t)
 %U_{su}(t(k+1)) \leq \max[U_{su}(t(k)) - \sum_{t=t(k)}^{t(k+1)-1} \mu_{su}(t), 0] + \sum_{t=t(k)}^{t(k+1)-1} R_{su}(t)
\label{eq:u_su}
\end{align}
where $R_{su}(t)$ denotes the number of new packets admitted in slot $t$
and $t_{k+1}$ denotes the start of the $(k+1)^{th}$ frame. 
The above expression has an inequality because it may be possible to
serve the packets admitted in the $k^{th}$ frame during that frame itself.

In order to meet the time average power constraint, we make use of a virtual power queue $X_{su}(t_k)$ \cite{neely-energy} which evolves 
over frames as follows:
\begin{align}
X_{su}(t_{k+1}) = \max[X_{su}(t_k) - T[k] P_{avg} + \sum_{t=t_k}^{t_{k+1}-1} P(t), 0]
\label{eq:x_su}
\end{align}
where $T[k] = t_{k+1} - t_k$ is the length of the $k^{th}$ frame. 
%\emph{Note that for all states $i \geq 1$, $C_{su}(t) \in \{C_{coop}, 0\}$ while for state $i = 0$, $C_{su}(t) \in \{C_{tx}, 0\}$.} 
Recall that $T[k]$ is a (random) function of the control decisions taken during the $k^{th}$ frame.

In order to construct an optimal dynamic control policy, we use the technique of \cite{chih-ping, asilomar10, neely_new} 
where a ratio of ``drift-plus-penalty'' is maximized over every 
frame. Specifically, let $\boldsymbol{Q}(t_k) = (Q_{su}(t_k), X_{su}(t_k))$ denote the queueing state of the system at the start of the $k^{th}$ frame.
As a measure of the congestion in the system, we use a Lyapunov function 
$L(\boldsymbol{Q}(t_k)) \defequiv \frac{1}{2}[Q_{su}^2(t_k) + X_{su}^2(t_k)]$. 
Define the drift  $\Delta(t_k)$ as the conditional expected change in 
$L(\boldsymbol{Q}(t_k))$ over the frame $k$: 
\begin{align}
\Delta(t_k) \defequiv \expect{L(\boldsymbol{Q}(t_{k+1})) - L(\boldsymbol{Q}(t_k)) | \boldsymbol{Q}(t_k)}
\label{eq:femto_drift_def}
\end{align}
Then, using (\ref{eq:u_su}) and (\ref{eq:x_su}), we can bound $\Delta(t_k)$ as follows:
\begin{align}
&\Delta(t_k) \leq B - Q_{su}(t_k) \expect{\sum_{t=t_k}^{t_{k+1}-1} [\mu_{su}(t) - R_{su}(t)] | \boldsymbol{Q}(t_k)} \nonumber\\
&\qquad - X_{su}(t_k) \expect{T[k] P_{avg} - \sum_{t=t_k}^{t_{k+1}-1} P(t) | \boldsymbol{Q}(t_k)}
\label{eq:femto_drift1}
\end{align}
where $B$ is a finite constant that satisfies the following for all $k$ and $\boldsymbol{Q}(t_k)$ under any control algorithm:
\begin{align*}
B \geq \frac{1}{2} \mathbb{E}\Bigg\{ &\Big(\sum_{t=t_k}^{t_{k+1}-1} \mu_{su}(t) \Big)^2 + \Big(\sum_{t=t_k}^{t_{k+1}-1} R_{su}(t) \Big)^2 \\
&+ \Big(\sum_{t=t_k}^{t_{k+1}-1} P(t) - T[k] P_{avg} \Big)^2| \boldsymbol{Q}(t_k) \Bigg\} 	
\end{align*}
Using the fact that $\mu_{su}(t) \leq \mu_{max}, P(t) \leq P_{max}$ for all $t$, and using the fact (\ref{eq:second_moment}), 
it follows that choosing $B$ as follows satisfies the above:
\begin{align}
B = \frac{D[\mu_{max}^2 + A_{max}^2 + (P_{max} - P_{avg})^2]}{2} 
\label{eq:femto_B}
\end{align}

Adding a penalty term $-V \expect{\sum_{t=t_k}^{t_{k+1}-1} R_{su}(t) | \boldsymbol{Q}(t_k)}$ (where $V > 0$ is a control parameter that affects a utility-delay
trade-off as shown in Theorem \ref{thm:femto_performance}) to both sides and rearranging yields:
\begin{align}
&\Delta(t_k) - V \expect{\sum_{t=t_k}^{t_{k+1}-1} R_{su}(t) | \boldsymbol{Q}(t_k)} \leq B + (Q_{su}(t_k) - V)\nonumber\\
 %&\qquad\qquad- U_{su}(k) \expect{\sum_{t=t[k]}^{t[k+1]-1} \mu_{su}(t) - R_{su}(t) | \boldsymbol{Q}(k)} \nonumber\\
 %&\qquad\qquad - X_{su}(k) \expect{T[k] P_{avg} - \sum_{t=t[k]}^{t[k+1]-1} P(t) | \boldsymbol{Q}(k)} \nonumber \\
 %&\qquad\qquad - V \expect{\sum_{t=t[k]}^{t[k+1]-1} R_{su}(t) | \boldsymbol{Q}(k)} \nonumber \\
& \times\expect{\sum_{t=t_k}^{t_{k+1}-1} R_{su}(t) | \boldsymbol{Q}(t_k)} - X_{su}(t_k) \expect{T[k] P_{avg} | \boldsymbol{Q}(t_k)} \nonumber\\
 %& - U_{su}(k) \expect{\sum_{t=t[k]}^{t[k+1]-1} \mu_{su}(t) | \boldsymbol{Q}(k)} \nonumber\\
 %&+ X_{su}(k) \expect{\sum_{t=t[k]}^{t[k+1]-1} P(t) | \boldsymbol{Q}(k)}
& - \expect{\sum_{t=t_k}^{t_{k+1}-1} \Big(Q_{su}(t_k) \mu_{su}(t) - X_{su}(t_k) P(t) \Big) | \boldsymbol{Q}(t_k)} 
\label{eq:femto_drift2}
\end{align}

Minimizing the ratio of an upper bound on the right hand side of the above expression and the expected frame length
over all control options leads to the following \emph{Frame-Based-Drift-Plus-Penalty-Algorithm}.
 %\emph{Renewal-Based-Drift-Plus-Penalty-Algorithm}. 
In each frame $k \in \{1, 2, 3, \ldots \}$, do the following:
\begin{enumerate}

\item \emph{Admission Control}: For all $t \in \{t_k, t_k + 1, \ldots, t_{k+1} -1 \}$, choose $R_{su}(t)$ as follows:
\begin{align}
R_{su}(t) = \left\{ \begin{array}{ll} A_{su}(t) & \textrm{if $Q_{su}(t) \leq V$} \\
0 & \textrm{else}
\end{array} \right.
\label{eq:femto_adm_ctrl}
\end{align}

 %if $U_{su}(t) > V$, then $R_{su}(t) = 0$ 
 %Else $R_{su}(t) = A_{su}(t)$.

\item \emph{Resource Allocation}: Choose a policy that maximizes the following ratio:
\begin{align}
%\frac{\expect{U_{su}(k) \sum_{t=t[k]}^{t[k+1]-1} \mu_{su}(t) - X_{su}(k) \sum_{t=t[k]}^{t[k+1]-1} P(t) | \boldsymbol{Q}(k)}}{\expect{T[k] | \boldsymbol{Q}(k)}}
\frac{\expect{\sum_{t=t_k}^{t_{k+1}-1} \Big(Q_{su}(t_k) \mu_{su}(t) - X_{su}(t_k) P(t) \Big) | \boldsymbol{Q}(t_k)}}{\expect{T[k] | \boldsymbol{Q}(t_k)}}
\label{eq:femto_resource_alloc}
\end{align}

Specifically, every slot $t$ of the frame, the policy observes the
queue values $Q_{su}(t_k)$ and $X_{su}(t_k)$ at the beginning of the frame 
and selects a secondary user power $P(t)$ subject to the constraint $P(t) \in \mathcal{P}$
and the constraint on transmitting own data vs. cooperation depending on whether
slot $t$ is in the ``PU Idle'' or ``PU Busy'' period of the frame.
This is done in such a way that the above frame-based ratio of expectations is maximized. ¬¨‚Ä†
Recall that the frame size $T[k]$ is influenced by the policy through the success probabilities
that are determined by secondary user power selections. Further recall that these success
probabilities are different during the ``PU Idle'' and ``PU Busy'' periods of the frame.
An explicit policy that maximizes this expectation is given in the next section.

\item \emph{Queue Update}: After implementing this policy, update the queues as in (\ref{eq:ch4_queues2}) and (\ref{eq:x_su}).
\end{enumerate}

From the above, it can be seen that the admission control part (\ref{eq:femto_adm_ctrl}) is a simple threshold-based decision
that does not require any knowledge of the arrival rates $\lambda_{su}$ or $\lambda_{pu}$. 
In the next section, we present an explicit solution to the maximizing
policy for the resource allocation in (\ref{eq:femto_resource_alloc}) and show that, remarkably, it
also does not require knowledge of $\lambda_{su}$ or $\lambda_{pu}$ and
can be computed easily.
We will then analyze the performance of the \emph{Frame-Based-Drift-Plus-Penalty-Algorithm} 
in Sec. \ref{section:femto_analysis}.

%We will show in Sec. \ref{section:solving} that the resource allocation part
%(\ref{eq:femto_resource_alloc}) can also be implemented without any knowledge of 
%$\lambda_{su}$ or $\lambda_{pu}$ and 
%admits a simple solution that does not require dynamic programming or
%explicit learning and is easy to compute. 

\section{The Maximizing Policy of (\ref{eq:femto_resource_alloc})}
\label{section:solving}

The policy that maximizes (\ref{eq:femto_resource_alloc})
uses only two numbers that we call $P_0^*$ and $P_1^*$,
defined as follows. $P_0^*$ is given by the solution to the
following optimization problem:
\begin{align}
 %\textrm{Maximize:} \;\; & U_{su}(k) \mathbb{E}\{\mu_{su}(C_0)\} - X_{su}(k) \expect{C_0} \nonumber \\
\textrm{Maximize:} \;\; & Q_{su}(t_k) \mu_{su}(P_0) - X_{su}(t_k) P_0 \nonumber \\
\textrm{Subject to:} \;\; & P_0 \in \mathcal{P}
\label{eq:P_0_define}
\end{align}
Let $\theta^* \defequiv Q_{su}(t_k) \mu_{su}(P_0^*) - X_{su}(t_k) P_0^*$ denote the value of the
objective of (\ref{eq:P_0_define}) under the optimal solution. Then, 
$P_1^*$ is given by the solution to the following optimization problem:
\begin{align}
\textrm{Minimize:} \;\; &\frac{ \theta^* + X_{su}(t_k) P_1 }{\phi(P_1)} \nonumber \\
\textrm{Subject to:} \;\;& P_1 \in \mathcal{P}
\label{eq:P_1_define}
\end{align}
Note that both (\ref{eq:P_0_define}) and (\ref{eq:P_1_define}) are simple optimization problems in a single variable and
can be solved efficiently.
Given $P_0^*$ and $P_1^*$, on every slot $t$ of frame $k$, the policy that maximizes (\ref{eq:femto_resource_alloc})  
chooses power $P(t)$ as follows:
\begin{align}
P(t)  = \left\{ \begin{array}{ll} P_0^* & \textrm{if $Q_{pu}(t) = 0$} \\
P_1^* & \textrm{if $Q_{pu}(t) > 0$}
\end{array} \right.
\label{eq:P_su}
\end{align}

That is, the secondary user uses the constant power $P_0^*$ for its own transmission 
during the ``PU Idle''period of the frame, and uses constant power $P_1^*$ for cooperative transmission
during all slots of the ``PU busy''period of the frame. 
Note that $P_0^*$ and $P_1^*$ can be computed easily based on the weights
$Q_{su}(t_k), X_{su}(t_k)$ associated with frame $k$, 
and do not require knowledge of the arrival rates $\lambda_{su}, \lambda_{pu}$.

Our proof that the above decisions maximize (\ref{eq:femto_resource_alloc})
 has the following parts: First, we show that the decisions that maximize the ratio of expectations in 
(\ref{eq:femto_resource_alloc}) are the same as the optimal decisions in an equivalent infinite
horizon ¬¨‚Ä†Markov decision problem (MDP). ¬¨‚Ä†Next, we show that the solution to the
infinite horizon MDP uses fixed power $P_i$ for each queue state $Q_{pu}(t) = i$ (for
$i \in \{0, 1, 2, \ldots\}$). ¬¨‚Ä†Then, we show that $P_i$ are the same for all $i \geq 1$. ¬¨‚Ä†Finally,
we show that the optimal powers $P_0^*$ and $P_1^*$ are given as above. The detailed proof is
given in the next section.

\subsection{Proof Details}
\label{section:solving_detail}

Recall that the \emph{Frame-Based-Drift-Plus-Penalty-Algorithm} 
chooses a policy that maximizes the following ratio over every  frame $k \in \{1, 2, 3, \ldots \}$
\begin{align}
 %\frac{\expect{U_{su}(k) \sum_{t=t[k]}^{t[k+1]-1} \mu_{su}(t) - X_{su}(k) \sum_{t=t[k]}^{t[k+1]-1} P_{su}(t) | \boldsymbol{Q}(k)}}{\expect{T[k] | \boldsymbol{Q}(k)}}
\frac{\expect{\sum_{t=t_k}^{t_{k+1}-1} \Big(Q_{su}(t_k) {\mu}_{su}(t) - X_{su}(t_k) {P}(t) \Big)| \boldsymbol{Q}(t_k)}}{\expect{T[k] | \boldsymbol{Q}(t_k)}}
\label{eq:resource_alloc1}
\end{align}
subject to the constraints described in Sec. \ref{section:femto_basic}.
Here we examine how to solve (\ref{eq:resource_alloc1}) in detail. 
First, define the state $i$ in any slot $t \in \{t_k, t_k + 1, \ldots, t_{k+1} -1 \}$
as the value of the primary user queue backlog $Q_{pu}(t)$ in that slot.
Now let $\mathcal{R}$ denote the class of stationary, randomized policies where every policy $r \in \mathcal{R}$ chooses a power allocation 
$P_i(r) \in \mathcal{P}$ in each state $i$ according to a stationary distribution. 
It can be shown that it is sufficient to only consider policies in $\mathcal{R}$ to maximize  (\ref{eq:resource_alloc1}).
Now suppose a policy $r \in \mathcal{R}$ 
is implemented on a \emph{recurrent} system with fixed $Q_{su}(t_k)$ and $X_{su}(t_k)$ and with the same state dynamics as our model. 
Note that $\mu_{su}(t) = 0$ for all $t$ when the state $i \geq 1$. 
Then, by basic renewal theory \cite{gallager}, 
we have that maximizing the ratio in (\ref{eq:resource_alloc1}) is equivalent to the following optimization problem:
\begin{align}
 %\textrm{Maximize:} \;\; &U_{su}(k) \expect{\mu_{su}(P_{0}(r))} \pi_0(r) - X_{su}(k) \expect{P_{0}(r)} \pi_0 (r) 
 %- X_{su}(k) \sum_{i \geq 1} \expect{C_{i}(r)} \pi_i(r) \nonumber\\
\textrm{Maximize:} \; &Q_{su}(t_k) \expect{\mu_{su}(P_{0}(r))} \pi_0(r) \nonumber \\
&- X_{su}(t_k) \sum_{i \geq 0} \expect{P_{i}(r)} \pi_i(r) \nonumber\\
\textrm{Subject to:}\; &r \in \mathcal{R}
\label{eq:femto_time_avg}
\end{align}
where $\pi_i(r)$ is the resulting steady-state probability of being in state $i$ in the recurrent system 
under the stationary, randomized policy $r$ and where the expectations
above are with respect to $r$. Note that well-defined steady-state probabilities $\pi_i(r)$ exist for all $r \in \mathcal{R}$ because
we have assumed that $\lambda_{pu} < \phi_{nc}$ so that even if no cooperation is used, the primary queue is stable and the system is recurrent.
Thus, solving (\ref{eq:resource_alloc1}) is equivalent to solving the 
\emph{unconstrained time average maximization problem} 
 %with objective 
(\ref{eq:femto_time_avg}) over the class of stationary, randomized policies. Note that (\ref{eq:femto_time_avg}) is an 
infinite horizon Markov decision problem (MDP) over the state space $i \in \{0, 1, 2, \ldots \}$.
We study this problem in the following.

Consider the optimal stationary, randomized policy that maximizes the objective in (\ref{eq:femto_time_avg}).
Let $\chi_i$ denote the probability distribution over $\mathcal{P}$ that is used by this policy to choose a power allocation $P_i$ in state $i$.
Let $\mu_i$ denote the resulting effective probability of successful primary transmission in state $i \geq 1$. Then we have that 
$\mu_i = \mathbb{E}_{\chi_i} \{\phi(P_i)\}$ where $\phi(P_i)$ denotes the probability of successful transmission in state $i$ when the secondary user spends power
$P_i$ in cooperative transmission with the primary user. 
Since the system is stable and has a well-defined steady-state distribution, we can write down the detail equations for the Markov
Chain that describes the state transitions of the system as follows (See Fig. \ref{fig:bd}):
\begin{align*}
\pi_0 \lambda_{pu} &= \pi_1 (1-\lambda_{pu}) \mu_1  \\
\pi_1 \lambda_{pu}(1- \mu_1) &= \pi_{2}(1-\lambda_{pu}) \mu_2 \\
&\vdots \\
\pi_i \lambda_{pu}(1- \mu_i) &= \pi_{i+1}(1-\lambda_{pu}) \mu_{i+1} \qquad \forall i \geq 1
\end{align*}
where $\pi_i$ denotes the steady-state probability of being in state $i$ under this policy.
Summing over all $i$ yields:
\begin{align}
\lambda_{pu} = \sum_{i\geq1} \pi_i \mu_i 
\label{eq:eq1_new}
\end{align}
The average power incurred in cooperative transmissions under this policy is given by:
\begin{align}
\overline{P} = \sum_{i \geq 1} \pi_i \mathbb{E}_{\chi_i} \{P_i\}
\label{eq:eq2_new}
\end{align}

\begin{figure}
\centering
\includegraphics[width=8.5cm]{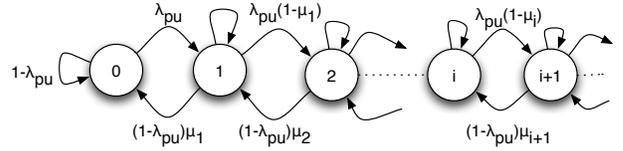}
\caption{Birth-Death Markov Chain over the system state where the system state represents the primary user queue backlog.}
\label{fig:bd}
\end{figure}

 %\begin{align}
 %\chi' = \chi(i) \; \textrm{with probability} \; \frac{\pi_i}{\sum_{i \geq 1} \pi_i} \; \forall i
 %\end{align}

Now consider an alternate stationary policy that uses the following fixed distribution $\chi '$ for choosing control action $P'$ in all states $i \geq 1$:
%\begin{displaymath}
\begin{align}
\chi'  \defequiv \left\{ \begin{array}{llll} \chi_1 & \textrm{with probability $\frac{\pi_1}{\sum_{j \geq 1} \pi_j}$} \\
\chi_2 & \textrm{with probability $\frac{\pi_2}{\sum_{j \geq 1} \pi_j}$} \\
\vdots  & \\
\chi_i & \textrm{with probability $\frac{\pi_i}{\sum_{j \geq 1} \pi_j}$} \\
\vdots  & 
\end{array} \right.
\label{eq:chi'}
\end{align}
%\end{displaymath}

Let $\mu'$ denote the resulting effective probability of a successful primary transmission in any state $i \geq 1$. Note that this is 
same for all states by the definition (\ref{eq:chi'}). Then, we have that:
\begin{align}
\mu' = \sum_{i \geq 1} \mu_i \frac{\pi_i}{\sum_{j \geq 1} \pi_j}
\end{align}
Let $\pi_i'$ denote the steady-state probability of being in state $i$ under this alternate policy.
Note that the system is stable under this alternate policy as well. Thus, using the detail equations 
for the Markov Chain that describes the state transitions of the system under this policy yields
\begin{align}
&\lambda_{pu} = \sum_{k \geq 1} \pi_k' \mu' = \sum_{k \geq 1} \pi_k' \Big(\sum_{i \geq 1} \mu_i \frac{\pi_i}{\sum_{j \geq 1} \pi_j} \Big) \nonumber\\
&= \sum_{k \geq 1} \pi_k' \Big( \frac{\sum_{i \geq 1} \mu_i \pi_i}{\sum_{j \geq 1} \pi_j} \Big) 
= \sum_{k \geq 1} \pi_k' \Big( \frac{\lambda_{pu}}{\sum_{j \geq 1} \pi_j} \Big)
\end{align}
where we used (\ref{eq:eq1_new}) in the last step.
This implies that $\sum_{k \geq 1} \pi_k' = \sum_{j \geq 1} \pi_j$ and therefore $\pi_0' = \pi_0$.
Also, the average power incurred in cooperative transmissions under this alternate policy is given by:
\begin{align}
&\overline{P}' = \sum_{k \geq 1} \pi_k' \mathbb{E}_{\chi'} \{P'\} 
= \sum_{k \geq 1} \pi_k'  \Big( \sum_{i \geq 1} \mathbb{E}_{\chi_i} \{P_i\} \frac{\pi_i}{\sum_{j \geq 1} \pi_j} \Big) \nonumber\\
&=  \sum_{k \geq 1} \pi_k' \Big( \frac{\overline{P}}{\sum_{j \geq 1} \pi_j} \Big) = \overline{P}
\end{align}
where we used (\ref{eq:eq2_new}) in the second last step and $\sum_{k \geq 1} \pi_k' = \sum_{j \geq 1} \pi_j$ in the last step.

Thus, if we choose $\chi' = \chi_0$ in state $i=0$ and choose $\chi'$ as defined in (\ref{eq:chi'}) in all other states, 
it can be seen that \emph{the alternate policy achieves the same time average
value of the objective (\ref{eq:femto_time_avg}) as the optimal policy}.
This implies that to maximize (\ref{eq:femto_time_avg}), it is sufficient to optimize over the class of stationary policies that use the 
\emph{same} distribution for choosing $P_i$ for all states $i \geq 1$. Denote this class by $\mathcal{R}'$. Then for all $i > 1$, 
we have that $\expect{P_i(r)} = \expect{P_1(r)}$ for all $r \in \mathcal{R}'$. 
Using this and the fact that $1 - \pi_0(r) = \sum_{i \geq 1} \pi_i(r)$, (\ref{eq:femto_time_avg}) can be simplified as follows:
\begin{align}
\textrm{Maximize:} \; & [Q_{su}(t_k) \mathbb{E}\{\mu_{su}(P_0(r))\} - X_{su}(t_k) \expect{P_0(r)}] \pi_0(r) \nonumber\\
& - X_{su}(t_k) \expect{P_1(r)} (1 - \pi_0(r)) \nonumber \\
\textrm{Subject to:} \; & r \in \mathcal{R}'
\label{eq:time_avg2}
\end{align}
where $\pi_0(r)$ is the resulting steady-state probability of being in state $0$ and where $\expect{P_1(r)}$ is the 
average power incurred in cooperative transmission in state $i = 1$ (same for all states $i \geq 1$).
Next, note that the control decisions taken by the secondary user in state $i = 0$ do not affect the length of the frame and therefore $\pi_0(r)$. 
 %Since the control decisions taken by the secondary user in state $i = 0$ do not affect the length of the renewal frame, 
Further, the expectations can be removed.
Therefore the first term in the problem above can be maximized separately as follows:
\begin{align}
 %\textrm{Maximize:} \;\; & U_{su}(k) \mathbb{E}\{\mu_{su}(C_0)\} - X_{su}(k) \expect{C_0} \nonumber \\
\textrm{Maximize:} \;\; & Q_{su}(t_k) \mu_{su}(P_0) - X_{su}(t_k) P_0 \nonumber \\
\textrm{Subject to:} \;\; & P_0 \in \mathcal{P}
\label{eq:time_avg3}
\end{align}
This is the same as (\ref{eq:P_0_define}). 
Let $P_0^*$ denote the optimal solution to (\ref{eq:time_avg3}) and let $\theta^* = Q_{su}(t_k) \mu_{su}(P_0^*) - X_{su}(t_k) P_0^*$ denote the value of the
objective of (\ref{eq:time_avg3}) under the optimal solution. Note that we must have that $\theta^* \geq 0$ because the value of the objective
when the secondary user chooses $P_0 = 0$ (i.e., stays idle) is $0$.
Then, (\ref{eq:time_avg2}) can be written as:
\begin{align}
\textrm{Maximize:} \;\; & \theta^* \pi_0(r) - X_{su}(t_k) \expect{P_1(r)} (1 - \pi_0(r)) \nonumber \\
\textrm{Subject to:} \;\; & r \in \mathcal{R}'
\label{eq:time_avg4}
\end{align}
The effective probability of a successful primary transmission in any state $i \geq 1$ is given by $\mathbb{E} \{\phi(P_1(r))\}$. 
Using Little's Theorem, we have $\pi_0(r) = 1 - \frac{\lambda_{pu}}{\mathbb{E} \{\phi(P_1(r))\}}$. Using this and rearranging
the objective in (\ref{eq:time_avg4}) and ignoring the constant terms, we have the following equivalent problem:
\begin{align}
\textrm{Minimize:} \;\; &\frac{ \theta^* + X_{su}(t_k) \mathbb{E} \{P_1(r)\} }{\mathbb{E} \{\phi(P_1(r))\}} \nonumber \\
\textrm{Subject to:} \;\;& r \in \mathcal{R}'
\label{eq:time_avg5}
\end{align}
It can be shown that it is sufficient to consider only \emph{deterministic} power allocations to solve (\ref{eq:time_avg5}) (see, for example,
\cite[Section 7.3.2]{neely_new}). This yields the following problem:
\begin{align}
\textrm{Minimize:} \;\; &\frac{ \theta^* + X_{su}(t_k) P_1 }{\phi(P_1)} \nonumber \\
\textrm{Subject to:} \;\;& P_1 \in \mathcal{P}
\label{eq:time_avg6}
\end{align}
This is the same as (\ref{eq:P_1_define}). 
Note that solving this problem does not require knowledge of $\lambda_{pu}$ or $\lambda_{su}$ 
and can be solved easily for general power allocation options $\mathcal{P}$. 
%We present two examples that admit particularly simple solutions to this problem.
We present an example that admits a particularly simple solution to this problem.

Suppose $\mathcal{P} = \{0, P_{max}\}$ so that the secondary user can either cooperate with full power $P_{max}$ or
not cooperate (with power expenditure  $0$) with the primary user. 
Then, the optimal solution to (\ref{eq:time_avg6}) can be calculated by comparing the value of 
its objective for $P_1 \in \{0, P_{max}\}$. This yields the following simple threshold-based rule:
\begin{align}
P_1^*  = \left\{ \begin{array}{ll} 0 & \textrm{if $X_{su}(t_k) \geq \frac{\theta^* (\phi_c - \phi_{nc})}{ P_{max} \phi_{nc}}$  } \\
P_{max} & \textrm{else} 
\end{array} \right.
\label{eq:P1*_1}
\end{align}
We also note that this threshold can be computed without any knowledge of the input rates $\lambda_{pu}, \lambda_{su}$.

\begin{comment}

\begin{enumerate}
\item Suppose $\mathcal{P} = \{0, P_{max}\}$ so that the secondary user can either cooperate with full power $P_{max}$ or
not cooperate (with power expenditure  $0$) with the primary user. 
Then, the optimal solution to (\ref{eq:time_avg6}) can be calculated by comparing the value of 
its objective for $P_1 \in \{0, P_{max}\}$. This yields the following simple threshold-based rule:
\begin{align}
P_1^*  = \left\{ \begin{array}{ll} 0 & \textrm{if $X_{su}(t_k) \geq \frac{\theta^* (\phi_c - \phi_{nc})}{ P_{max} \phi_{nc}}$  } \\
P_{max} & \textrm{else} 
\end{array} \right.
\label{eq:P1*_1}
\end{align}
We also note that this threshold can be computed without any knowledge of the input rates $\lambda_{pu}, \lambda_{su}$.

\item  Suppose $\mathcal{P} = [0, P_{max}]$ and $\phi(P)$ is concave and increasing in $P$ for all $P \in [0, P_{max}]$. Then it can be shown
that $\frac{ \theta^* + X_{su}(t_k) P }{\phi(P)}$ is convex in $P$. Let $\hat{P}$ denote the solution to the following:
\begin{align*}
\phi'(P) = \frac{ X_{su}(t_k) \phi(P)}{\theta^* + X_{su}(t_k) P}
\end{align*}
where $\phi'(P)$ is the derivative of $\phi(P)$. Then we have the following solution to (\ref{eq:time_avg6}):
\begin{align}
P_1^*  = \left\{ \begin{array}{lll} 0 & \textrm{if $\hat{P} < 0$} \\
\hat{P} & \textrm{if $0 \leq \hat{P} \leq P_{max}$} \\ 
P_{max} & \textrm{else} \\ 
\end{array} \right.
\label{eq:P1*_2}
\end{align}
\end{enumerate}

\end{comment}

To summarize, the overall solution to (\ref{eq:femto_resource_alloc}) is given by the pair $(P_0^*, P_1^*)$ where $P_0^*$ denotes the 
power allocation used by the secondary user for its own transmission when the primary user is idle 
and $P_1^*$ denotes the power used by the secondary user for cooperative
transmission. Note that these values remain fixed for the entire duration of frame $k$. 
However, these can change from one frame to another depending on the values of the
queues $Q_{su}(t_k), X_{su}(t_k)$. The computation of $(P_0^*, P_1^*)$
can be carried out using a two-step process as follows:

\begin{enumerate}
\item First, compute $P_0^*$ by solving problem (\ref{eq:time_avg3}). Let $\theta^*$ be the value of the
objective of (\ref{eq:time_avg3}) under the optimal solution $P_0^*$. 
\item Then compute $P_1^*$ by solving problem (\ref{eq:time_avg6}).
\end{enumerate}

It is interesting to note that in order to implement this algorithm, the secondary user does not require knowledge of the current
queue backlog value of the primary user. Rather, it only needs to know the values of its own queues and
whether the current slot is in the ``PU Idle'' or
``PU Busy'' part of the frame. This is quite different from the conventional solution to the MDP (\ref{eq:cmdp}) which is
typically a different randomized policy for each value of the state (i.e., the primary queue backlog).

\section{Performance Analysis}
\label{section:femto_analysis}

%Consider any algorithm that chooses a control policy $c[k] \in \mathcal{C}$ at the beginning of each renewal frame 
%$k \in \{1, 2, 3, \ldots\}$. 

To analyze the performance of 
 the \emph{Frame-Based-Drift-Plus-Penalty-Algorithm}, 
we compare its Lyapunov drift with that of the optimal stationary, randomized policy \emph{STAT} of Lemma \ref{lem:femto_one}. 
First, note that by basic renewal theory \cite{gallager}, the performance guarantees provided by \emph{STAT}
hold over every frame $k \in \{1, 2, 3, \ldots \}$. Specifically, let $t_k$ be the start of the $k^{th}$ frame. Suppose 
\emph{STAT} is implemented over this frame. Then the following hold:
\begin{align}
& \expect{\sum_{t=t_k}^{\hat{t}_{k+1}-1} R^{stat}_{su}(t)} = \expect{\hat{T}[k]} \upsilon^* 
\label{eq:iid_1} \\
& \expect{\sum_{t=t_k}^{\hat{t}_{k+1}-1} R^{stat}_{su}(t)} \leq \expect{\sum_{t=t_k}^{\hat{t}_{k+1}-1} \mu^{stat}_{su}(t)} 
\label{eq:iid_2} \\
& \expect{ \sum_{t=t_k}^{\hat{t}_{k+1}-1} P^{stat}_{su}(t)} \leq \expect{\hat{T}[k]} P_{avg} 
\label{eq:iid_3}
\end{align}
where $\hat{t}_{k+1}$ and $\hat{T}[k]$ denote the start of the $(k+1)^{th}$ frame and the length of
the $k^{th}$ frame, respectively, under the policy \emph{STAT}. 
Similarly, $R^{stat}_{su}(t), {P}_{su}^{stat}(t), {\mu}_{su}^{stat}(t)$ denote the resource allocation decisions 
under \emph{STAT}.

Next, we define an alternate control algorithm \emph{ALT} that will be useful in analyzing the
performance of 
the \emph{Frame-Based-Drift-Plus-Penalty-Algorithm}.

\emph{Algorithm ALT:} In each frame $k \in \{1, 2, 3, \ldots \}$, do the following:
\begin{enumerate}

\item \emph{Admission Control}: For all $t \in \{t_k, t_k + 1, \ldots, t_{k+1} -1 \}$, choose $R_{su}(t)$ as follows:
\begin{align}
R_{su}(t) = \left\{ \begin{array}{ll} A_{su}(t) & \textrm{if $Q_{su}(t_k) \leq V$} \\
0 & \textrm{else}
\end{array} \right.
\label{eq:adm_ctrl_alt}
\end{align}

\item \emph{Resource Allocation}: Choose a policy that maximizes the following ratio:
\begin{align}
 %\frac{\expect{U_{su}(k) \sum_{t=t[k]}^{t[k+1]-1} \mu_{su}(t) - X_{su}(k) \sum_{t=t[k]}^{t[k+1]-1} P_{su}(t) | \boldsymbol{Q}(k)}}{\expect{T[k] | \boldsymbol{Q}(k)}}
\frac{\expect{\sum_{t=t_k}^{t_{k+1}-1} \Big(Q_{su}(t_k) \mu_{su}(t) - X_{su}(t_k) P(t) \Big) | \boldsymbol{Q}(t_k)}}{\expect{T[k] | \boldsymbol{Q}(t_k)}}
\label{eq:resource_alloc_alt}
\end{align}
\item \emph{Queue Update}: After implementing this policy, update the queues as in (\ref{eq:u_su}), (\ref{eq:x_su}).
\end{enumerate}

By comparing with 
the \emph{Frame-Based-Drift-Plus-Penalty-Algorithm}, 
it can be see that 
this algorithm differs only in the admission control part while the resource allocation decisions are exactly the same. 
Specifically, under \emph{ALT}, the queue backlog $Q_{su}(t_k)$ at the start of the $k^{th}$ frame is used for making admission control decisions for
the entire duration of that frame. However, under the \emph{Frame-Based-Drift-Plus-Penalty-Algorithm}, 
the queue backlog $Q_{su}(t)$ at the start of \emph{each} slot is used for
making admission control decisions. 
Note that since the length of the frame depends only on the resource allocation decisions and they are the same under the two algorithms,
it follows that implementing them with the same starting backlog $\boldsymbol{Q}(t_k)$ yields the same frame lengths.

The following lemma compares the value of the second term in the Lyapunov drift bound (\ref{eq:femto_drift2}) that
corresponds to the admission control decisions under these two algorithms. 
\begin{lem} 
Let $R^{fab}_{su}(t)$ and $R^{alt}_{su}(t)$ denote the admission control decisions made by 
 the \emph{Frame-Based-Drift-Plus-Penalty-Algorithm} and the \emph{ALT} algorithm respectively for
 %the algorithms \emph{FAB} and \emph{ALT} respectively for
all $t \in \{t_k, t_k + 1, \ldots, t_{k+1} -1 \}$. Then we have:
\begin{align}
&\expect{\sum_{t=t_k}^{t_{k+1}-1} (Q_{su}(t_k) - V) R^{alt}_{su}(t) | \boldsymbol{Q}(t_k)} \nonumber \\
&\geq \expect{\sum_{t=t_k}^{t_{k+1}-1} (Q_{su}(t_k) - V) R^{fab}_{su}(t) | \boldsymbol{Q}(t_k)} - C
\label{eq:adm_ctrl_bound}
\end{align}
 %where $C \defequiv \frac{T_{max}(T_{max}-1)}{2} (A_{max} + \mu_{max})A_{max}$ is a constant that does not depend on $V$.
where $C \defequiv \frac{D(A_{max} + \mu_{max})A_{max}}{2}$ is a constant that does not depend on $V$.
\label{lem:adm_ctrl_bound_lem}
\end{lem} 

\begin{proof} See Appendix A.
\end{proof}

\begin{comment}

The following lemma compares the value of the second term in the Lyapunov drift bound (\ref{eq:femto_drift2}) that
corresponds to the admission control decisions under these two algorithms. 
Its proof is given in Appendix \ref{section:C_1}.

%the \emph{Renewal-Based-Drift-Plus-Penalty-Algorithm} and the \emph{ALT} algorithm.
%\emph{FAB} and \emph{ALT}.

\begin{lem} 
Let $R^{fab}_{su}(t)$ and $R^{alt}_{su}(t)$ denote the admission control decisions made by 
 the \emph{Frame-Based-Drift-Plus-Penalty-Algorithm} and the \emph{ALT} algorithm respectively for
 %the algorithms \emph{FAB} and \emph{ALT} respectively for
all $t \in \{t_k, t_k + 1, \ldots, t_{k+1} -1 \}$. Then we have:
\begin{align}
\expect{\sum_{t=t_k}^{t_{k+1}-1} (Q_{su}(t_k) - V) R^{alt}_{su}(t) | \boldsymbol{Q}(t_k)} 
\geq \expect{\sum_{t=t_k}^{t_{k+1}-1} (Q_{su}(t_k) - V) R^{fab}_{su}(t) | \boldsymbol{Q}(t_k)} - C
\label{eq:adm_ctrl_bound}
\end{align}
 %where $C \defequiv \frac{T_{max}(T_{max}-1)}{2} (A_{max} + \mu_{max})A_{max}$ is a constant that does not depend on $V$.
where $C \defequiv \frac{D(A_{max} + \mu_{max})A_{max}}{2}$ is a constant that does not depend on $V$.
\label{lem:adm_ctrl_bound_lem}
\end{lem} 

%\begin{proof} See Appendix \ref{section:C_1}.
%\end{proof}

\end{comment}

%\subsection{Performance Theorem}
%\label{section:theorem}

We are now ready to characterize the performance of 
the \emph{Frame-Based-Drift-Plus-Penalty-Algorithm}.
%\emph{FAB}.

\begin{thm} (Performance Theorem) Suppose 
 the \emph{Frame-Based-Drift-Plus-Penalty-Algorithm} 
 %the control algorithm \emph{FAB}
is implemented over all frames $k \in \{1, 2, 3, \ldots \}$ with initial condition $Q_{su}(0) = 0, X_{su}(0) = 0$
and with a control parameter $V > 0$. 
Let $\mu_{su}^{fab}(t), P_{su}^{fab}(t)$ denote the resource allocation decisions
under this algorithm. 
Then, we have:
\begin{enumerate}
\item The secondary user queue backlog $Q_{su}(t)$ is upper bounded for all $t$:
\begin{align}
Q_{su}(t) \leq Q_{max} \defequiv A_{max} + V
\label{eq:u_max}
\end{align}

\item The virtual power queue $X_{su}(t_k)$ is mean rate stable, i.e.,
\begin{align}
\lim_{K\rightarrow\infty} \frac{\expect{X_{su}(t_K)}}{K} = 0
\label{eq:x_max1}
\end{align}
Further, we have:
\begin{align}
 %&\limsup_{K\rightarrow\infty} \Bigg( \frac{1}{K} \sum_{k=1}^K \expect{\sum_{t=t[k]}^{t[k+1]-1} P_{su}(t)} - P_{avg} \frac{1}{K} \sum_{k=1}^K \expect{T[k]} \Bigg) \leq 0 
 %\label{eq:x_max2} \\
&\limsup_{K\rightarrow\infty} \Bigg( \frac{1}{K} \sum_{k=1}^K \expect{\sum_{t=t_k}^{t_{k+1}-1} (P_{su}^{fab}(t) - P_{avg})} \Bigg) \leq 0 
\label{eq:x_max2} \\
&\limsup_{K\rightarrow\infty} \frac{\frac{1}{K} \sum_{k=1}^K \expect{\sum_{t=t_k}^{t_{k+1}-1} P_{su}^{fab}(t)}}{\frac{1}{K} \sum_{k=1}^K \expect{T[k]}} \leq P_{avg}
\label{eq:x_max3}
\end{align}

\item The time-average secondary user throughput (defined over frames) satisfies the following bound for all $K > 0$:
 %\begin{align}
 %\frac{1}{K} \sum_{k=1}^K \expect{\sum_{t=t[k]}^{t[k+1]-1} R_{su}^{fab}(t)} 
 %\geq  r^* \frac{1}{K} \sum_{k=1}^K \expect{T[k]} - \frac{B + C}{V}
  %\frac{\expect{L(\boldsymbol{Q}(K+1))} - \expect{L(\boldsymbol{Q}(1))}}{KV} 
 %\label{eq:utility_bound}
%\end{align}
\begin{align}
\frac{\sum_{k=1}^K \expect{\sum_{t=t_k}^{t_{k+1}-1} R_{su}^{fab}(t)}} {\sum_{k=1}^K \expect{T[k]}} 
\geq  \upsilon^* - \frac{B + C}{V T_{min}}
 %\frac{\expect{L(\boldsymbol{Q}(K+1))} - \expect{L(\boldsymbol{Q}(1))}}{KV} 
\label{eq:femto_utility_bound}
\end{align}
where $B = \frac{D[\mu_{max}^2 + A_{max}^2 + (P_{max} - P_{avg})^2]}{2}$ and 
 %$C= \frac{T_{max}(T_{max}-1)}{2} (A_{max} + \mu_{max})A_{max}$.
$C= \frac{D(A_{max} + \mu_{max})A_{max}}{2}$.
 
\end{enumerate}
\label{thm:femto_performance}
\end{thm}
Theorem \ref{thm:femto_performance} shows that the time-average secondary user throughput can be pushed to
within $O(1/V)$ of the optimal value with a trade-off in the worst case queue backlog.
By Little's Theorem, this leads to an $O(1/V, V)$ utility-delay tradeoff.

\begin{proof} 
{Part (1)}: We argue by induction. First, note that (\ref{eq:u_max}) holds for $t =0$. Next, suppose $Q_{su}(t) \leq Q_{max}$ for
some $t > 0$. We will show that $Q_{su}(t+1) \leq Q_{max}$. We have two cases. First, suppose $Q_{su}(t) \leq V$. Then,
by (\ref{eq:u_su}), the maximum that $Q_{su}(t)$ can increase is $A_{max}$ so that $Q_{su}(t+1) \leq A_{max} + V = Q_{max}$.
Next, suppose $Q_{su}(t) > V$. Then, the admission control decision (\ref{eq:femto_adm_ctrl}) chooses $R_{su}(t) = 0$.
Thus, by (\ref{eq:u_su}), we have that $Q_{su}(t+1) \leq Q_{su}(t) \leq Q_{max}$ for this case as well. Combining these two
cases proves the bound (\ref{eq:u_max}). 

{Parts (2) and (3)}: See Appendix B. 
\end{proof}

%\section{Continuous Time Markov Chain Example}
%TBD

\section{Extensions to Basic Model}
\label{section:femto_extensions}

We consider two extensions to the basic model of Sec. \ref{section:femto_basic}.

\subsection{Multiple Secondary Users}
\label{section:multiple_su}

Consider the scenario with one primary user as before, but with $N > 1$ secondary users. The primary user
channel occupancy process evolves as before where the secondary users can transmit their own data only when
the primary user is idle. However, they may cooperatively transmit with the primary user to increase its transmission success
probability. In general, multiple secondary users may cooperatively transmit with the primary in one timeslot.
However, for simplicity, here we assume that at most one secondary user can take part in a cooperative transmission per slot.
Further, we also assume that at most one secondary user can transmit its data when the primary user is idle.
 
Our formulation can be easily extended to this scenario. Let $\mathcal{P}_i$ denote the set of power allocation options
for secondary user $i$. Suppose each secondary user $i$ is subject to
average and peak power constraints $P_{avg, i}$ and $P_{max, i}$ respectively. 
Also, let $\phi_i(P)$ denote the success probability of the primary transmission when
secondary user $i$ spends power $P$ in cooperative transmission.
Now consider the objective of maximizing the sum total throughput of the secondary users subject to each user's 
average and peak power constraints and the scheduling constraints of the model.
In order to apply the ``drift-plus-penalty'' ratio method, we use the following queues:
\begin{align}
 %&U_{su, i}(k+1) \leq \max[U_{su, i}(k) - \sum_{t=t[k]}^{t[k+1]-1} \mu_{su, i}(t), 0] + \sum_{t=t[k]}^{t[k+1]-1} R_{su, i}(t)
&Q_{i}(t_{k+1}) \leq \max[Q_{i}(t_k) - \sum_{t=t_k}^{t_{k+1}-1} \mu_{i}(t), 0] + \sum_{t=t_k}^{t_{k+1}-1} R_{i}(t)
\label{eq:u_su_i} \\
 %&X_{su, i}(k+1) = \max[X_{su, i}(k) - T[k] P_{avg, i} + \sum_{t=t[k]}^{t[k+1]-1} P_{su, i}(t), 0]
&X_{i}(t_{k+1}) = \max[X_{i}(t_k) - T[k] P_{avg, i} + \sum_{t=t_k}^{t_{k+1}-1} P_{i}(t), 0]
\label{eq:x_su_i}
\end{align}
where $Q_{i}(t_k)$ is the queue backlog of secondary user $i$ at the beginning of the $k^{th}$ frame,
$\mu_{i}(t)$ is the service rate of secondary user $i$ in slot $t$,
$R_{i}(t)$ and $P_{i}(t)$ denote the number of new packets admitted and the power expenditure incurred by the
secondary user $i$ in slot $t$. Finally, $t_{k+1}$ denotes the start of the $(k+1)^{th}$ frame and 
 $T[k] = t_{k+1} - t_k$ is the length of the $k^{th}$ frame as before.

%where $U_{su, i}(k)$ is the queue backlog of secondary user $i$ at the beginning of the $k^{th}$ frame,
%$\mu_{su,i}(t)$ is the service rate of secondary user $i$ in slot $t$,
%$R_{su, i}(t)$ and $P_{su, i}(t)$ denote the number of new packets admitted and the power expenditure incurred by the
%secondary user $i$ in slot $t$. Finally, $t[k+1]$ denotes the start of the $(k+1)^{th}$ renewal frame and 
% $T[k] = t[k+1] - t[k]$ is the length of the $k^{th}$ renewal frame as before.

Let $\boldsymbol{Q}(t_k) = (Q_{1}(t_k), \ldots, Q_{N}(t_k), X_{1}(t_k), \ldots, X_{N}(t_k))$ 
denote the queueing state of the system at the start of the $k^{th}$ frame.
Using a Lyapunov function $L(\boldsymbol{Q}(t_k)) \defequiv \frac{1}{2}\Big[\sum_{i=1}^N Q_{i}^2(t_k) + \sum_{i=1}^N X_{i}^2(t_k)\Big]$ and
following the steps in Sec. \ref{section:solution} yields the following 
\emph{Multi-User Frame-Based-Drift-Plus-Penalty-Algorithm}.
 %\emph{Renewal-Based-Drift-Plus-Penalty-Algorithm}. 
In each frame $k \in \{1, 2, 3, \ldots \}$, do the following:
\begin{enumerate}

\item \emph{Admission Control}: For all $t \in \{t_k, t_k + 1, \ldots, t_{k+1} -1 \}$, for each secondary user $i \in \{1, 2, \ldots, N\}$,
choose $R_{i}(t)$ as follows:
\begin{align}
R_{i}(t) = \left\{ \begin{array}{ll} A_{i}(t) & \textrm{if $Q_{i}(t) \leq V$} \\
0 & \textrm{else}
\end{array} \right.
\label{eq:adm_ctrl_i}
\end{align}
where $A_i(t)$ is the number of new arrivals to secondary user $i$ in slot $t$.

 %if $U_{su}(t) > V$, then $R_{su}(t) = 0$ 
 %Else $R_{su}(t) = A_{su}(t)$.

\item \emph{Resource Allocation}: Choose a policy that maximizes the following ratio:
\begin{align}
 %\frac{\sum_{i=1}^N \expect{U_{su, i}(k) \sum_{t=t[k]}^{t[k+1]-1} \mu_{su, i}(t) - X_{su, i}(k) \sum_{t=t[k]}^{t[k+1]-1} P_{su, i}(t) 
 %| \boldsymbol{Q}(k)}}{\expect{T[k] | \boldsymbol{Q}(k)}}
\frac{\sum_{i=1}^N \expect{\sum_{t=t_k}^{t_{k+1}-1} (Q_i(t_k) \mu_{i}(t) - X_{i}(t_k) P_{i}(t)) | \boldsymbol{Q}(t_k)}}{\expect{T[k] | \boldsymbol{Q}(t_k)}}
\label{eq:resource_alloc_i}
\end{align}

\item \emph{Queue Update}: After implementing this policy, update the queues as in (\ref{eq:u_su_i}) and (\ref{eq:x_su_i}).
\end{enumerate}
Similar to the basic model, this algorithm can be implemented without any knowledge of the arrival rates $\lambda_{i}$ 
or $\lambda_{pu}$. Further, using the techniques developed in Sec. \ref{section:solving}, it can be shown that the solution to
(\ref{eq:resource_alloc_i}) can be computed in two steps as follows.
First, we solve the following problem for each $i \in \{1, 2, \ldots, N\}$:
\begin{align}
 %\textrm{Maximize:} \;\; & U_{su,i}(k) \mu_{su,i}(P) - X_{su, i}(k) P \nonumber \\
\textrm{Maximize:} \;\; & Q_{i}(t_k) \mu_{i}(P) - X_{i}(t_k) P \nonumber \\
\textrm{Subject to:} \;\; & P \in \mathcal{P}_i
\label{eq:time_avg3_i}
\end{align}
Let $P_0^*$ denote the optimal solution to (\ref{eq:time_avg3_i}) achieved by user $i^*$
and let $\theta^*$ denote the optimal objective value. 
This means user $i^*$ transmits on
all idle slots of frame $k$ with power $P_0^*$. Next, to determine the optimal cooperative transmission strategy, 
we solve the following problem for each $i \in \{1, 2, \ldots, N\}$:
\begin{align}
 %\textrm{Minimize:} \;\; &\frac{ \theta^* + X_{su, i}(k) P }{\phi_i(P)} \nonumber \\
\textrm{Minimize:} \;\; &\frac{ \theta^* + X_{i}(t_k) P }{\phi_i(P)} \nonumber \\
\textrm{Subject to:} \;\;& P \in \mathcal{P}_i
\label{eq:time_avg6_i}
\end{align}
Let $P_1^*$ denote the optimal solution to (\ref{eq:time_avg6_i}) achieved by user $j^*$.
This means user $j^*$ cooperatively transmits on
all busy slots of frame $k$ with power $P_1^*$.

\subsection{Fading Channels}
\label{section:fading}

Next, suppose there is an additional channel fading process $S(t)$ that takes values from a finite set $\mathcal{S}$ in an i.i.d fashion every slot. 
We assume that in every slot, $\textrm{Prob}[S(t) = s] = q_s$ for all $s \in \mathcal{S}$.
The success probability with cooperative transmission now is a function of both the power allocation and the fading state in that slot.
Specifically, suppose the primary user is active in slot $t$ and the secondary user allocates power $P(t)$ for cooperative transmission. Also suppose
$S(t) = s$. Then the random success/failure outcome of the primary transmission is given by an indicator variable $\mu_{pu}(P(t), s)$ and the 
success probability is given by $\phi_s(P(t)) = \expect{\mu_{pu}(P(t), s)}$. The function $\phi_s(P)$ is known to the network controller
for all $s \in \mathcal{S}$ and is assumed to be non-decreasing in $P$ for each $s \in \mathcal{S}$. 
For simplicity, we assume that the secondary user transmission rate $\mu_{su}(t)$ depends only on $P(t)$.

By applying the  ``drift-plus-penalty'' ratio method to this extended model, we get the following control algorithm. The
admission control remains the same as (\ref{eq:femto_adm_ctrl}). The resource allocation part involves maximizing the ratio in
(\ref{eq:femto_resource_alloc}). Using the same arguments as before in Sec. \ref{section:solving}, it can be shown that
maximizing this ratio is equivalent to the following optimization problem:
\begin{align}
 %\textrm{Maximize:} \;\; &Q_{su}(k) \expect{\mu_{su}(P_{0}(r))} \pi_0(r) - X_{su}(k) \expect{P_{0}(r)} \pi_0 (r) 
 %- X_{su}(k) \sum_{i \geq 1} \expect{C_{i}(r)} \pi_i(r) \nonumber\\
&\textrm{Max:} \; Q_{su}(t_k) \expect{\mu_{su}(P_{0}(r))} \pi_0(r) - X_{su}(t_k) \expect{P_{0}(r)} \pi_0(r) \nonumber\\
&\qquad - X_{su}(t_k) \sum_{i \geq 1} \sum_{s \in \mathcal{S}} \expect{P_{i, s}(r)} \pi_{i, s}(r) \nonumber\\
& \textrm{Subject to:}\; r \in \mathcal{R}
\label{eq:time_avg_fading}
\end{align}
where $\pi_{i, s}(r)$ is the resulting steady-state probability of being in state $(i, s)$ in the recurrent system 
under the stationary, randomized policy $r$ and where the expectations
above are with respect to $r$. We study this problem in the following.

  %Then we have that 
  %$\mu_{i, s} = \mathbb{E}_{\chi_{i, s}} \{\phi_s(P_{i,s})\}$.
  %where $\phi_s(P_{i,s})$ denotes the probability of successful transmission in state $(i, s)$ 
  %when the secondary user spends power $P_{i,s}$ in cooperative transmission with the primary user. 
 
 %Using this, the detail equations for the Markov
 %Chain that describes the state transitions of the recurrent system under this policy are given by (See Fig. \ref{fig:bd}):
 %\begin{align*}
 %\pi_0 \lambda_{pu} &= \pi_1 (1-\lambda_{pu}) \mu_1  \\
 %\pi_1 \lambda_{pu}(1- \mu_1) &= \pi_{2}(1-\lambda_{pu}) \mu_2 \\
 %&\vdots \\
 %\pi_i \lambda_{pu}(1- \mu_i) &= \pi_{i+1}(1-\lambda_{pu}) \mu_{i+1} \qquad \forall i \geq 1
 %\end{align*}
 %where $\pi_i$ denotes the steady-state probability of being in state $i$ under this policy.
 %Summing over all $i$ yields:

Consider the optimal stationary, randomized policy that maximizes the objective in (\ref{eq:time_avg_fading}).
Let $\chi_{i,s}$ denote the probability distribution over $\mathcal{P}$ that is used by this policy to choose 
a control action $P_{i, s}$ in state $(i, s)$. Let $\mu_{i,s} = \mathbb{E}_{\chi_{i, s}} \{\phi_s(P_{i,s})\}$ 
denote the resulting effective probability of successful primary transmission in state $(i, s)$ where $i \geq 1$. 
Since the system is stable under any stationary policy, $\textrm{total incoming rate} = \textrm{total outgoing rate}$. 
Thus, we get:
\begin{align}
\lambda_{pu} = \sum_{i\geq1} \sum_{s \in \mathcal{S}} \pi_{i, s} \mu_{i,s} 
\label{eq:eq1_new_fading}
\end{align}
where $\pi_{i, s}$ denotes the steady-state probability of being in state $(i,s)$ under this policy.
 %These are well-defined since the system is stable.
Note that the system is stable and has a well-defined steady-state distribution.
The average power incurred in cooperative transmissions under this policy is given by:
\begin{align}
\overline{P} = \sum_{i \geq 1} \sum_{s \in \mathcal{S}} \pi_{i,s} \mathbb{E}_{\chi_{i,s}} \{P_{i,s}\}
\label{eq:eq2_new_fading}
\end{align}

Now consider an alternate stationary policy that, for each $s \in \mathcal{S}$, 
uses the following fixed distribution $\chi'_s$ for choosing control action $P'_s$ in all states $(i, s)$ where $i \geq 1$:
\begin{align}
\chi'_s  \defequiv \left\{ \begin{array}{llll} \chi_{1,s} & \textrm{with probability $\frac{\pi_{1,s}}{\sum_{j \geq 1} \pi_{j,s}}$} \\
\chi_{2,s} & \textrm{with probability $\frac{\pi_{2,s}}{\sum_{j \geq 1} \pi_{j,s}}$} \\
\vdots  & \\
\chi_{i,s} & \textrm{with probability $\frac{\pi_{i,s}}{\sum_{j \geq 1} \pi_{j,s}}$} \\
\vdots  & 
\end{array} \right.
\label{eq:chi_fading}
\end{align}
For each $s \in \mathcal{S}$, let $\mu'_s$ denote the resulting effective probability of a successful primary transmission 
in any state $(i, s)$ where $i \geq 1$ under this policy. Note that this is same for all states $(i, s)$ where $i \geq 1$ by the 
definition (\ref{eq:chi_fading}). Then, we have that:
\begin{align}
\mu'_s = \sum_{i \geq 1} \mu_{i,s} \frac{\pi_{i,s}}{\sum_{j \geq 1} \pi_{j, s}}
\label{eq:mu_s_fading}
\end{align}
Let $\pi_{i,s}'$ denote the steady-state probability of being in state $(i,s)$ under this alternate policy.
Since the system is stable under any stationary policy, total incoming rate = total outgoing rate. Thus, we get:
\begin{align}
\lambda_{pu} &= \sum_{s \in \mathcal{S}} \sum_{k \geq 1} \pi_{k,s}' \mu'_s = \sum_{s \in \mathcal{S}} \mu'_s \Bigg( \sum_{k \geq 1} \pi_{k,s}' \Bigg) \nonumber\\
&= \sum_{s \in \mathcal{S}} \Bigg[ \sum_{i \geq 1} \mu_{i,s} \frac{\pi_{i,s}}{\sum_{j \geq 1} \pi_{j, s}} \Bigg] \Bigg( \sum_{k \geq 1} \pi_{k,s}' \Bigg)
 %\sum_{i \geq 1} \pi_i' \Big(\sum_{i \geq 1} \mu_i \frac{\pi_i}{\sum_{i \geq 1} \pi_i} \Big) 
 %= \sum_{i \geq 1} \pi_i' \Big( \frac{\sum_{i \geq 1} \mu_i \pi_i}{\sum_{i \geq 1} \pi_i} \Big) 
 %= \sum_{i \geq 1} \pi_i' \Big( \frac{\lambda_{pu}}{\sum_{i \geq 1} \pi_i} \Big)
\label{eq:ipop_fading}
\end{align}
where we used (\ref{eq:mu_s_fading}) in the last step.
Since $S(t)$ is i.i.d., for any $s_1, s_2 \in \mathcal{S}$, we have that 
\begin{align*}
&\pi_0 q_{s1} +  \sum_{j \geq 1} \pi_{j, s1} = q_{s1}, \qquad \pi_0 q_{s2} +  \sum_{j \geq 1} \pi_{j, s2} = q_{s2} 
\end{align*}
Similarly, we have:
\begin{align*}
&\pi_0' q_{s1} +  \sum_{j \geq 1} \pi'_{j, s1} = q_{s1}, \qquad \pi_0' q_{s2} +  \sum_{j \geq 1} \pi'_{j, s2} = q_{s2} 
\end{align*}
Using this, for any $s_1, s_2 \in \mathcal{S}$, we have:
\begin{align}
\frac{\sum_{j \geq 1} \pi_{j, s1} }{\sum_{j \geq 1} \pi'_{j, s1}} = \frac{\sum_{j \geq 1} \pi_{j, s2} }{\sum_{j \geq 1} \pi'_{j, s2}} 
 %\frac{\sum_{j \geq 1} \pi_{j, s1} }{\sum_{j \geq 1} \pi_{j, s2}} = \frac{\sum_{j \geq 1} \pi'_{j, s1} }{\sum_{j \geq 1} \pi'_{j, s2}} = 
 %\frac{q_{s1}}{q_{s2}}
\end{align}
Using this in (\ref{eq:ipop_fading}), we have for each $\hat{s} \in \mathcal{S}$:
\begin{align}
\lambda_{pu} = \Bigg[ \sum_{s \in \mathcal{S}} \sum_{i \geq 1} \mu_{i,s} \pi_{i,s} \Bigg] \frac{\sum_{k \geq 1} \pi_{k,\hat{s}}'}{\sum_{j \geq 1} \pi_{j, \hat{s}}}  
= \lambda_{pu} \frac{\sum_{k \geq 1} \pi_{k,\hat{s}}'}{\sum_{j \geq 1} \pi_{j, \hat{s}}}
\end{align}
where we used (\ref{eq:eq1_new_fading}) in the last step.
This implies that $\sum_{k \geq 1} \pi_{k, \hat{s}}' = \sum_{j \geq 1} \pi_{j, \hat{s}}$ for every $\hat{s} \in \mathcal{S}$ 
and therefore $\pi_0' = \pi_0$.
Also, the average power incurred in cooperative transmissions under this alternate policy is given by:
\begin{align}
\overline{P}' &= \sum_{k \geq 1} \sum_{s \in \mathcal{S}} \pi_{k,s}' \mathbb{E}_{\chi'_s} \{P'_s\} \nonumber\\
&= \sum_{k \geq 1} \sum_{s \in \mathcal{S}} \pi_{k,s}'  
\Bigg( \sum_{i \geq 1} \mathbb{E}_{\chi_{i,s}} \{P_{i,s}\} \frac{\pi_{i,s}}{\sum_{j \geq 1} \pi_{j,s}} \Bigg) \nonumber\\
&=  \sum_{s \in \mathcal{S}}  \sum_{i \geq 1} \mathbb{E}_{\chi_{i,s}} \{P_{i,s}\} \pi_{i,s}
= \overline{P}
\end{align}
where we used the fact that $\sum_{k \geq 1} \pi_{k,s}' = \sum_{j \geq 1} \pi_{j,s}$ for all $s$.
Thus, if we choose $\chi' = \chi_0$ in state $i=0$ and choose $\chi_s'$ as defined in (\ref{eq:chi_fading}) in all states $(i, s)$ where $i \geq 1$, 
it can be seen that the alternate policy achieves the same time average
value of the objective (\ref{eq:time_avg_fading}) as the optimal policy.
This implies that to maximize (\ref{eq:time_avg_fading}), it is sufficient to optimize over the class of stationary policies that,
for each $s \in \mathcal{S}$, use the \emph{same} distribution for choosing $P_{i,s}$ for all states $(i, s)$ where $i \geq 1$. 
Denote this class by $\mathcal{R}'$. Using this and the fact that $\sum_{i \geq 1} \pi_{i,s}(r) = (1 - \pi_0(r))q_s$ for all $s$, 
(\ref{eq:time_avg_fading}) can be simplified as follows:
\begin{align}
\textrm{Maximize:} \;& [Q_{su}(t_k) \mathbb{E}\{\mu_{su}(P_0(r))\} - X_{su}(t_k) \expect{P_0(r)}] \pi_0(r) \nonumber\\
&- X_{su}(t_k) \sum_{s \in \mathcal{S}}  \expect{P_{s}(r)} (1 - \pi_0(r))q_s \nonumber \\
\textrm{Subject to:} \; & r \in \mathcal{R}'
\label{eq:time_avg2_fading}
\end{align}
where $\pi_0(r)$ is the resulting steady-state probability of being in state $0$ and where $\expect{P_{s}(r)}$ is the 
average power incurred in cooperative transmission in any state $(i,s)$ with $i \geq 1$.
Using the same arguments as before, the solution to (\ref{eq:time_avg2_fading}) can be obtained in two steps as follows.
We first compute the solution to (\ref{eq:time_avg3}) as before. Denoting its optimal value by $\theta^*$,
(\ref{eq:time_avg2_fading}) can be written as:
\begin{align}
\textrm{Maximize:} \;\; & \theta^* \pi_0(r) - X_{su}(t_k) \sum_{s \in \mathcal{S}} \expect{P_{s}(r)} (1 - \pi_0(r)) q_{s} \nonumber \\
\textrm{Subject to:} \;\; & r \in \mathcal{R}'
\label{eq:time_avg4_fading}
\end{align}

 %The effective probability of a successful primary transmission in any state $i \geq 1$ is given by $\mathbb{E} \{\phi(P_1(r))\}$. 

Using Little's Theorem, we have $\pi_0(r) = 1 - \frac{\lambda_{pu}}{\sum_{s \in \mathcal{S}} q_s \mathbb{E} \{\phi_s(P_{s}(r))\}}$. Using this and rearranging
the objective in (\ref{eq:time_avg4_fading}) and ignoring the constant terms, we have the following equivalent problem:
\begin{align}
\textrm{Maximize:} \;\; &\frac{ - \theta^* - X_{su}(t_k) \sum_{s \in \mathcal{S}} q_s \mathbb{E} \{P_{s}(r)\}}
   {\sum_{s \in \mathcal{S}} q_s \mathbb{E} \{\phi_s(P_{s}(r))\}} \nonumber \\
\textrm{Subject to:} \;\;& r \in \mathcal{R}'
\label{eq:time_avg5_fading}
\end{align}
It can be shown that it is sufficient to consider only \emph{deterministic} power allocations to solve (\ref{eq:time_avg5_fading}) (see, for example,
\cite[Section 7.3.2]{neely_new}). This yields the following 
problem:
\begin{align}
\textrm{Maximize:} \;\; &\frac{ - \theta^* - X_{su}(t_k) \sum_{s \in \mathcal{S}} q_s P_{s}}{\sum_{s \in \mathcal{S}} q_s \phi_s(P_{s})} \nonumber \\
\textrm{Subject to:} \;\;& P_{s} \in \mathcal{P} \; \textrm{for all} \; s \in \mathcal{S} 
\label{eq:time_avg6_fading}
\end{align}
Note that solving this problem does not require knowledge of $\lambda_{pu}$ or $\lambda_{su}$ 
and can be solved efficiently for general power allocation options $\mathcal{P}$.

%\section{Minimize Delay s.t. Energy Constraints}

\section{Simulations}
\label{section:femto_sim}

In this section, we evaluate the performance of the
\emph{Frame-Based-Drift-Plus-Penalty-Algorithm} using simulations.
We consider the network model as discussed in Sec. \ref{section:femto_basic}  with
one primary and one secondary user. The set $\mathcal{P}$ consists of only two options 
$\{0, P_{max} \}$.  We assume that $P_{avg} = 0.5$ and $P_{max} = 1$. 
We set $\phi_{nc} = 0.6$ and $\phi_c = 0.8$. For simplicity, we assume that
$\mu_{su}(P_{max}) = 1$.

\begin{figure}
\centering
\includegraphics[width=8cm, height=5cm, angle=0]{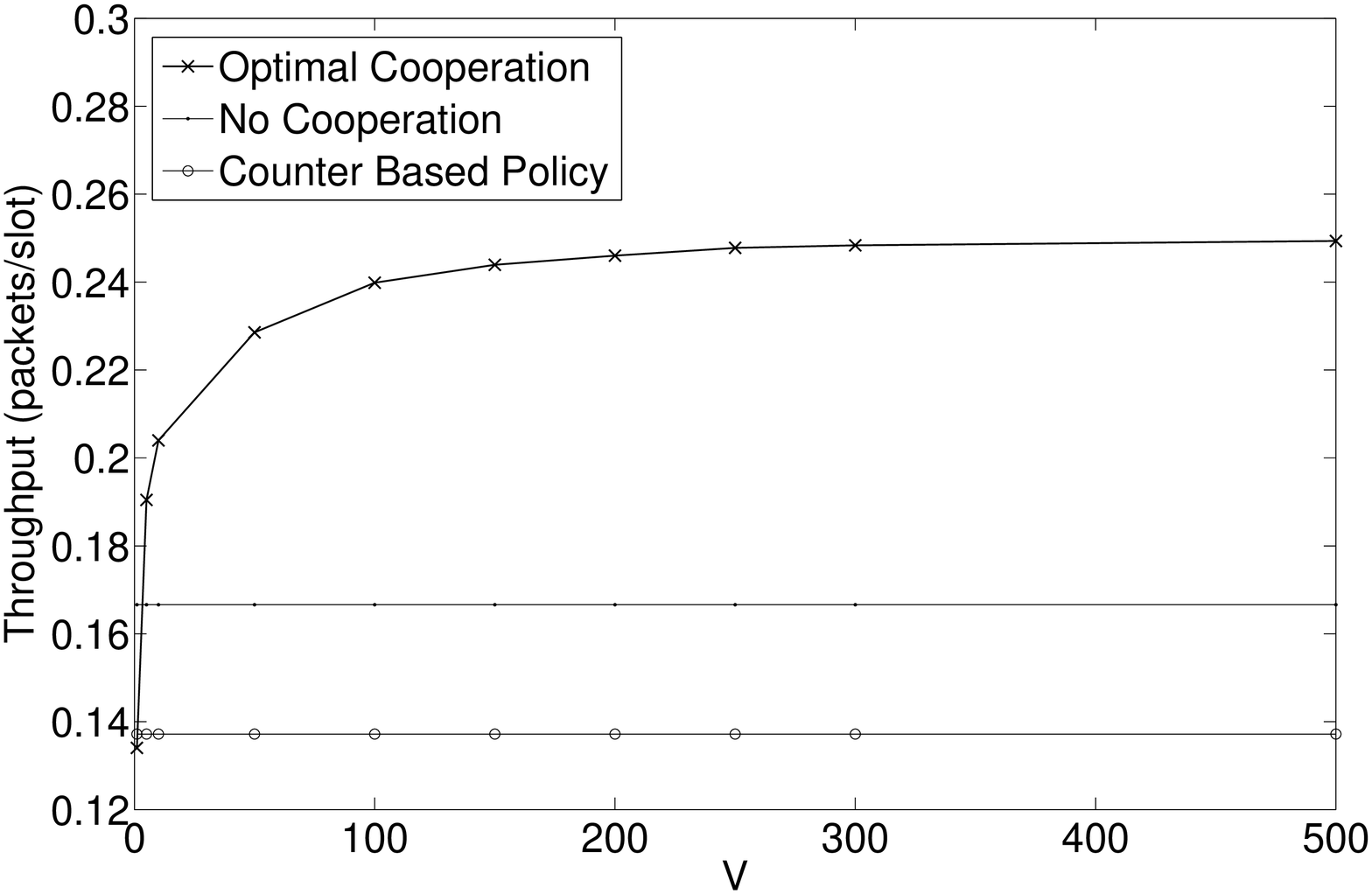}
\caption{Average Secondary User Throughput vs. V.}
\label{fig:femto_sim1}
\vspace{-0.2in}
\end{figure}

In the first set of simulations, we fix the input rates 
$\lambda_{pu} = \lambda_{su} = 0.5$ packets/slot. 
For these parameters, we can compute the optimal offline solution by linear programming. This yields the
maximum secondary user throughput as $0.25$ packets/slot.
We now simulate the \emph{Frame-Based-Drift-Plus-Penalty-Algorithm} for different values of the control
parameter $V$ over $1000$ frames.
In Fig. \ref{fig:femto_sim1}, we plot the average throughput achieved by the secondary user over this period. 
It can be seen that the average throughput increases with $V$ and converges to
the optimal value $0.25$ packets/slot, with the difference exhibiting a $O(1/V)$ behavior as predicted by Theorem \ref{thm:femto_performance}. 
In Fig. \ref{fig:femto_sim2}, we plot the average queue backlog of the secondary user over this period. It can be see that
the average queue backlog grows linearly in $V$, again  as predicted by Theorem \ref{thm:femto_performance}.
Also, for all $V$, the average secondary user power consumption over this period was found not to exceed $P_{avg} = 0.5$ units/slot.

For comparison, we also simulate three alternate algorithms. In the first algorithm ``No Cooperation'', the secondary user
never cooperates with the primary user and only attempts to maximize its throughput over the resulting idle periods.
The secondary user throughput under this algorithm was found to be $0.166$ packets/slot as shown in Fig. \ref{fig:femto_sim1}.
Note that using Little's Theorem, the resulting fraction of time the primary user is idle is 
$1 - {\lambda_{pu}}/{\phi_{nc}} = 1 - {0.5}/{0.6} = 0.166$. This limits the maximum secondary user throughput under the
``No Cooperation'' case to $0.166$ packets/slot.

In the second algorithm, we consider the ``Always Cooperate'' case where the secondary user
always cooperates with the primary user. For the example under consideration, this uses up all the
secondary user power and thus, the secondary user achieves zero throughput.

In the third algorithm ``Counter Based Policy'', a running average of the total secondary user power 
consumption so far is maintained. In each slot, the secondary user decides to transmit/cooperate only if
this running average is smaller than $P_{avg}$. The maximum secondary user throughput under this
algorithm was found to be $0.137$ packets/slot. This demonstrates that simply satisfying the
average power constraint is not sufficient to achieve maximum throughput. For example, it may be the case
that under the ``Counter Based Policy'', the running average condition is usually satisfied when the primary user is
busy. This causes the secondary user to cooperate. However, by the time the primary user next becomes idle, 
the running average exceeds $P_{avg}$ so that the secondary user does not transmit its own data. In contrast, the
\emph{Frame-Based-Drift-Plus-Penalty-Algorithm} is able to find the opportune moments to cooperate/transmit optimally.

\begin{figure}
\centering
\includegraphics[width=8cm, height=5cm, angle=0]{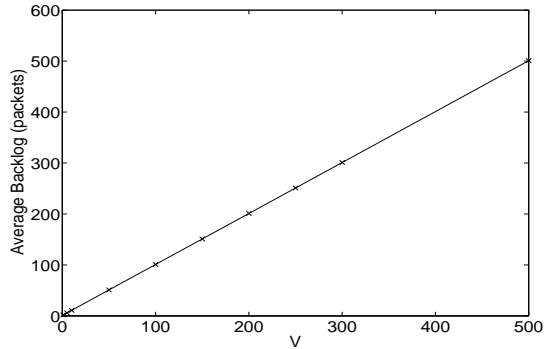}
\caption{Average Secondary User Queue Occupancy vs. V.}
\label{fig:femto_sim2}
\vspace{-0.2in}
\end{figure}

\begin{comment}

-- Also, the virtual queues were observed to be ....(verify this)

--- Note in Sims: Seem to requre large V...

-- Also Note about how running avg thing behaves when PU rate high --- possibly becauue of large B and need large V?

-- Also, verify what these values are when no non-ergodic changes

\end{comment}

In the second set of simulations, we fix the input rate $\lambda_{su} = 0.8$ packets/slot, $V=500$, and
simulate the \emph{Frame-Based-Drift-Plus-Penalty-Algorithm} over $1000$ frames.
At the start of the simulation, we set $\lambda_{pu} = 0.4$ packets/slot. The values of the other parameters
remain the same. However, during the course of the simulation, we change 
$\lambda_{pu}$ to $0.2$ packets/slot after the first $350$ frames and then again to 
$0.55$ packets/slot after the first $700$ frames. In Figs. \ref{fig:femto_sim3} and \ref{fig:femto_sim4}, 
we plot the running average (over $100$ frames) of the secondary user throughput and the average power used for
cooperation. These show that the \emph{Frame-Based-Drift-Plus-Penalty-Algorithm} automatically
adapts to the changes in $\lambda_{pu}$. Further, it quickly approaches the optimal performance corresponding to
the new $\lambda_{pu}$ by adaptively spending more or less power (as required) on cooperation. For example,
when $\lambda_{pu}$ reduces to $0.2$ packets/slot after frame number $350$, the fraction of time the primary is
idle even with no cooperation is $1 - 0.2/0.6 = 0.66$. With $P_{avg}=0.5$, there is no need to cooperate anymore.
This is precisely what the \emph{Frame-Based-Drift-Plus-Penalty-Algorithm} does as shown in Fig. \ref{fig:femto_sim4}.
Similarly, when when $\lambda_{pu}$ increases to $0.55$ packets/slot after frame number $700$, 
the \emph{Frame-Based-Drift-Plus-Penalty-Algorithm} starts to spend more power on cooperative transmissions.

\section{Conclusions}
%\section{Chapter Summary}
\label{section:conc}

In this paper, we studied the problem of opportunistic cooperation in a cognitive femtocell network.
Specifically, we considered the scenario where a secondary user
can cooperatively transmit with the primary user to increase
its transmission success probability. In return, the secondary
user can get more opportunities for transmitting its own data
when the primary user is idle. A key feature of this problem  
is that here, the evolution of the system state depends on the control actions taken by the secondary user. 
This dependence makes it a constrained Markov Decision Problem traditional solutions to which
require either extensive knowledge of the system dynamics or learning based approaches that suffer
from large convergence times. However, using the technique
of Lyaunov optimization, we designed a novel greedy and online
control algorithm that overcomes these challenges and is provably
optimal.

\begin{figure}
\centering
\includegraphics[width=8cm,  angle=0]{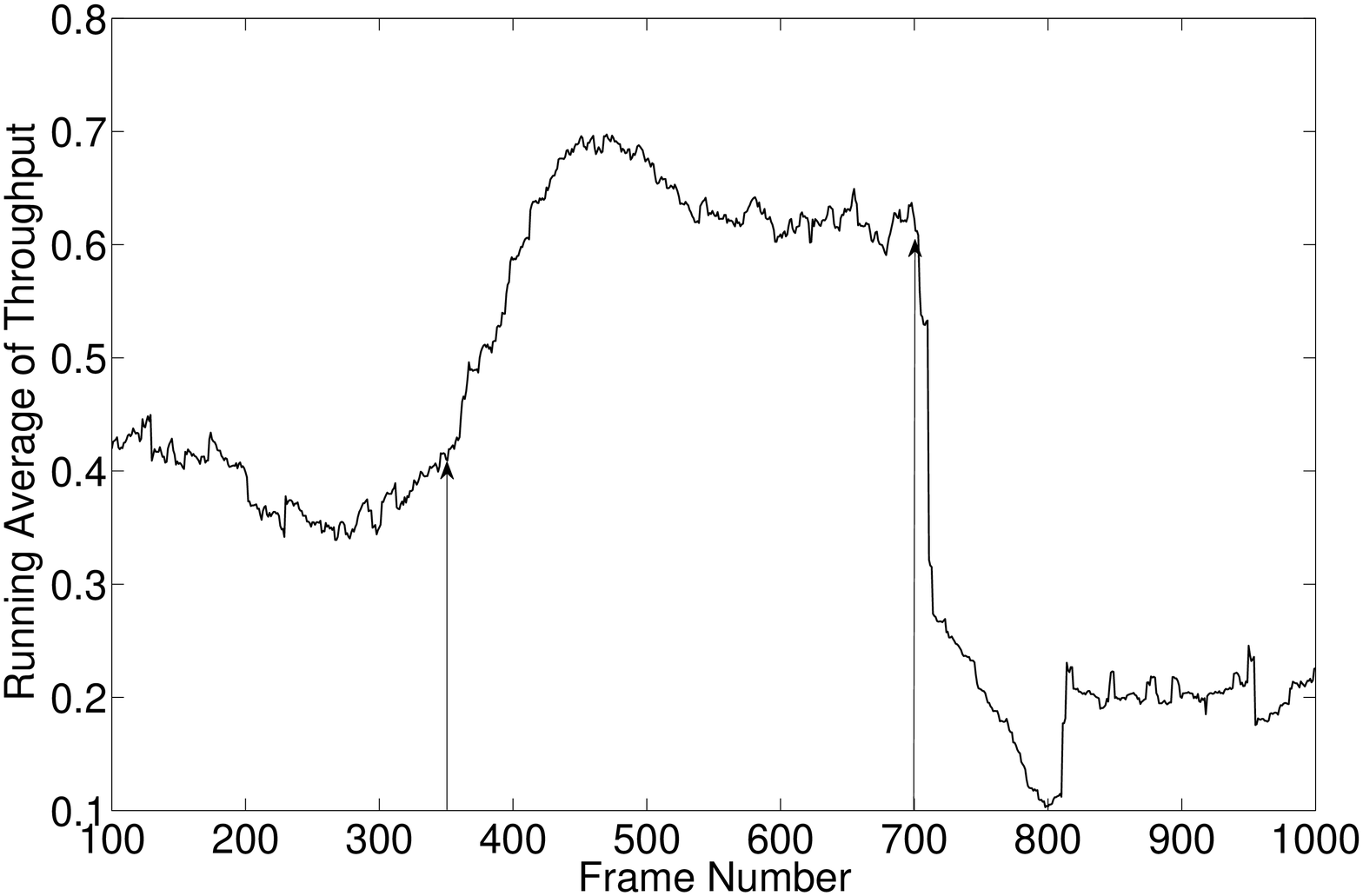}
\caption{Moving Average of Secondary User Throughput over Frames.}
\label{fig:femto_sim3}
\vspace{-0.2in}
\end{figure}

%\section{Proof of Lemma \ref{lem:adm_ctrl_bound_lem}}
%\label{section:C_1}

\section*{Appendix A \\ Proof of Lemma \ref{lem:adm_ctrl_bound_lem}}

Let $Q_{su}^{fab}(t)$ denote the queue backlog value under
the \emph{Frame-Based-Drift-Plus-Penalty-Algorithm} for all $t \in \{t_k, t_k+1, \ldots, t_{k+1}-1\}$.
Then, since the admission control decision (\ref{eq:femto_adm_ctrl}) 
 of the \emph{Frame-Based-Drift-Plus-Penalty-Algorithm} 
minimizes the term $(Q_{su}(t) - V) R_{su}(t)$ for all $Q_{su}(t)$, we have:

\begin{align}
&\expect{\sum_{t=t_k}^{t_{k+1}-1} (Q^{fab}_{su}(t) - V) R^{alt}_{su}(t) | \boldsymbol{Q}(t_k)}  \nonumber\\
&\geq \expect{\sum_{t=t_k}^{t_{k+1}-1} (Q^{fab}_{su}(t) - V) R^{fab}_{su}(t) | \boldsymbol{Q}(t_k)} 
\label{eq:ineq2_new}
\end{align}
Note that we are not implementing the admission control decisions of \emph{ALT} in the left hand side of the above.

Next, we make use of the following sample path relations in (\ref{eq:ineq2_new}) 
to prove (\ref{eq:adm_ctrl_bound}).
For all $t \in \{t_k, t_k+1, \ldots, t_{k+1}-1\}$, the following hold under any control algorithm:
\begin{align}
& Q_{su}(t_k) \geq Q_{su}(t) - (t - t_k)A_{max} 
\label{eq:sample1} \\
& Q_{su}(t_k) \leq Q_{su}(t) + (t - t_k)\mu_{max} 
\label{eq:sample2} 
\end{align}
(\ref{eq:sample1}) follows by noting that the maximum number of arrivals to the secondary user queue in
the interval $[t_k, \ldots, t)$ is at most $(t - t_k)A_{max}$. 
Similarly, (\ref{eq:sample2}) follows by noting that the maximum number of departures from the secondary user queue in
the interval $[t_k, \ldots, t)$ is at most $(t - t_k)\mu_{max}$.

\begin{figure}
\centering
\includegraphics[width=8cm,  angle=0]{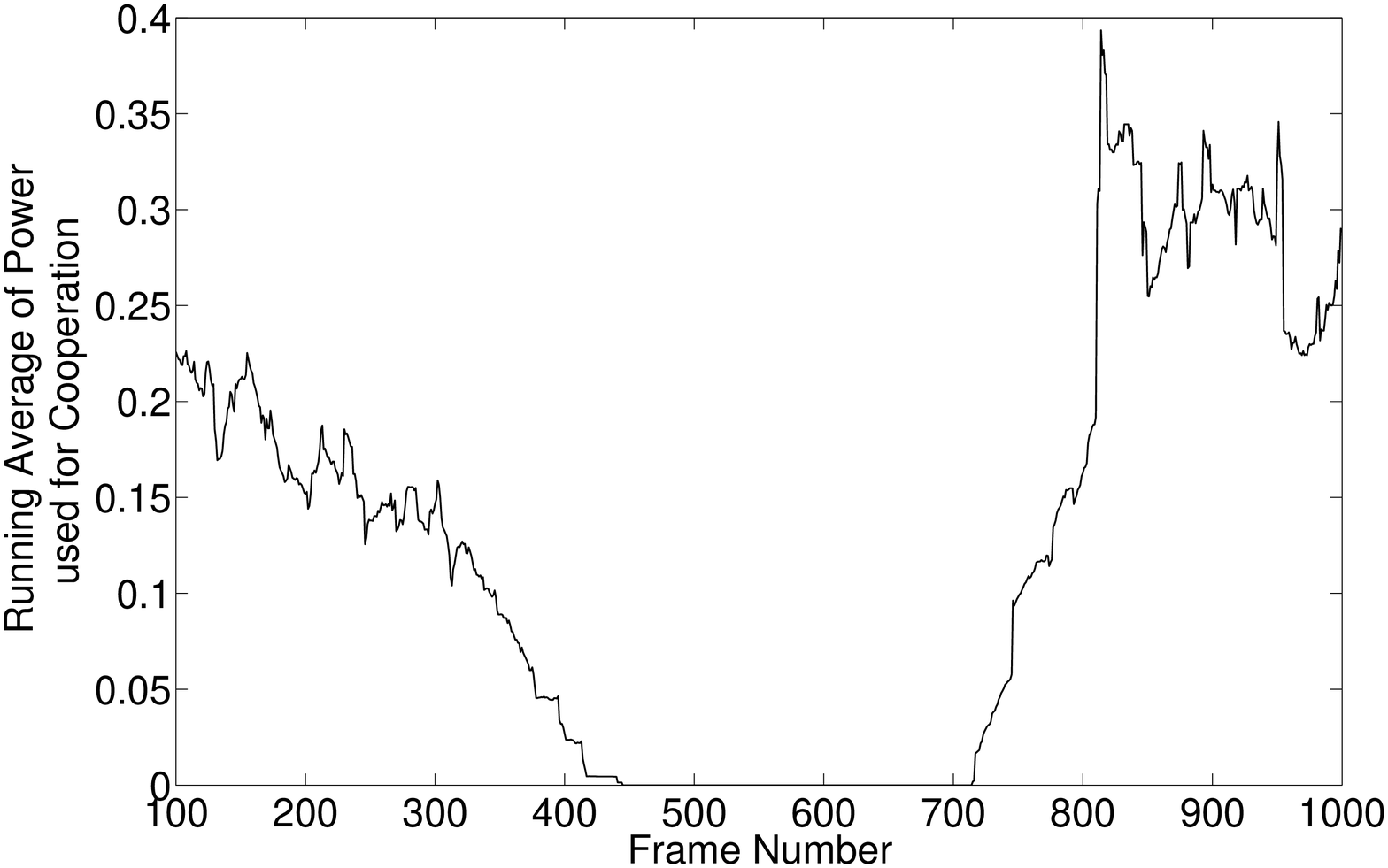}
\caption{Moving Average of Power used by the Secondary User for Cooperative Transmissions over Frames.}
\label{fig:femto_sim4}
\vspace{-0.2in}
\end{figure}

Using (\ref{eq:sample1}) in the left hand side of (\ref{eq:ineq2_new}) yields:
\begin{align*}
&\expect{\sum_{t=t_k}^{t_{k+1}-1} (Q_{su}^{fab}(t) - V) R_{su}^{alt}(t) | \boldsymbol{Q}(t_k)} \leq \\
&\expect{\sum_{t=t_k}^{t_{k+1}-1} (Q_{su}(t_k) - V) R^{alt}_{su}(t) | \boldsymbol{Q}(t_k)} \\
&+ \expect{\sum_{t=t_k}^{t_{k+1}-1} (t - t_k) A_{max} R_{su}^{alt}(t) | \boldsymbol{Q}(t_k)} 
 %& \expect{\sum_{t=t[k]}^{t[k+1]-1} (Q_{su}^{alt}(t) - V) R_{su}^{alt}(t) | \boldsymbol{Q}(k)} 
  %- \expect{\sum_{t=t[k]}^{t[k+1]-1} (t - t[k]) A^2_{max} | \boldsymbol{Q}(k)} = \\
 %& \expect{\sum_{t=t[k]}^{t[k+1]-1} (U_{su}^{alt}(t)-V) R_{su}^{alt}(t) | \boldsymbol{Q}(k)} - \expect{\frac{T[k](T[k]-1)}{2} A^2_{max} | \boldsymbol{Q}(k)} \geq\\
 %& \expect{\sum_{t=t[k]}^{t[k+1]-1} (U_{su}^{alt}(t) -V) R_{su}^{alt}(t) | \boldsymbol{Q}(k)} - \frac{T_{max}(T_{max}-1)}{2} A^2_{max}
\end{align*}
 %where $Q^{alt}_{su}(t)$ denotes the queue backlog when the policy \emph{ALT} is implemented over frame $k$.
Using the fact that $R_{su}^{alt}(t) \leq A_{max}$ and $\sum_{t=t_k}^{t_{k+1}-1} (t - t_k) =\frac{T[k](T[k]-1)}{2}$, we get:
\begin{align}
&\expect{\sum_{t=t_k}^{t_{k+1}-1} (Q^{fab}_{su}(t) - V) R^{alt}_{su}(t) | \boldsymbol{Q}(t_k)} \leq \nonumber\\
&\expect{\sum_{t=t_k}^{t_{k+1}-1} (Q_{su}(t_k) - V) R^{alt}_{su}(t) | \boldsymbol{Q}(t_k)} 
 %- \frac{D - T_{min}}{2} A^2_{max}
+ \frac{DA^2_{max}}{2} 
\label{eq:ineq1_new}
\end{align}

Next, using (\ref{eq:sample2}) in the right hand side of (\ref{eq:ineq2_new}) yields:
\begin{align*}
& \expect{\sum_{t=t_k}^{t_{k+1}-1} (Q_{su}^{fab}(t) - V) R^{fab}_{su}(t) | \boldsymbol{Q}(t_k)} \geq \\
& \expect{\sum_{t=t_k}^{t_{k+1}-1} (Q_{su}(t_k) - V) R_{su}^{fab}(t) | \boldsymbol{Q}(t_k)} \\
& - \expect{\sum_{t=t_k}^{t_{k+1}-1} (t - t_k)\mu_{max}  R_{su}^{fab}(t) | \boldsymbol{Q}(t_k)}
 %& \expect{\sum_{t=t[k]}^{t[k+1]-1} (U_{su}(k) - V) R_{su}^{fab}(t) | \boldsymbol{Q}(k)} 
 % - \expect{\sum_{t=t[k]}^{t[k+1]-1} (t - t[k]) \mu_{max} A_{max} | \boldsymbol{Q}(k)} = \\
 %& \expect{\sum_{t=t[k]}^{t[k+1]-1} (U_{su}(k)-V) R_{su}^{fab}(t) | \boldsymbol{Q}(k)} 
 % - \expect{\frac{T[k](T[k]-1)}{2} \mu_{max} A_{max} | \boldsymbol{Q}(k)} \geq\\
 %& \expect{\sum_{t=t[k]}^{t[k+1]-1} (U_{su}(k) -V) R_{su}^{fab}(t) | \boldsymbol{Q}(k)} - \frac{T_{max}(T_{max}-1)}{2} \mu_{max} A_{max}
\end{align*}
%where $U^{alt}_{su}(t)$ denotes the queue backlog under \emph{ALT}....
Again using the fact that $R_{su}^{fab}(t) \leq A_{max}$ and $\sum_{t=t_k}^{t_{k+1}-1} (t - t[k]) =\frac{T[k](T[k]-1)}{2}$, we get:
\begin{align}
&\expect{\sum_{t=t_k}^{t_{k+1}-1} (Q_{su}^{fab}(t) - V) R^{fab}_{su}(t) | \boldsymbol{Q}(t_k)} \geq \nonumber\\
&\expect{\sum_{t=t_k}^{t_{k+1}-1} (Q_{su}(t_k) - V) R^{fab}_{su}(t) | \boldsymbol{Q}(t_k)} 
 %- \frac{D - T_{min}}{2} \mu_{max} A_{max} 
- \frac{D\mu_{max} A_{max}}{2} 
\label{eq:ineq3_new}
\end{align}

Using (\ref{eq:ineq1_new}) and (\ref{eq:ineq3_new}) in (\ref{eq:ineq2_new}), we have:

\begin{align*}
&\expect{\sum_{t=t_k}^{t_{k+1}-1} (Q_{su}(t_k) - V) R^{alt}_{su}(t) | \boldsymbol{Q}(t_k)} \geq \nonumber \\
&\expect{\sum_{t=t_k}^{t_{k+1}-1} (Q_{su}(t_k) - V) R^{fab}_{su}(t) | \boldsymbol{Q}(t_k)} - C
 %\frac{D(A_{max}+\mu_{max})A_{max}}{2}
\end{align*}

\section*{Appendix B \\ Proof of Theorem \ref{thm:femto_performance}, parts 2 and 3}

 %\emph{Also, a solution always exists to (\ref{eq:resource_alloc}) because of the renewal assumptions...}

We prove parts (2) and (3) of Theorem \ref{thm:femto_performance} using the technique of Lyapunov optimization. 
Using (\ref{eq:femto_drift2}), a bound on the Lyapunov drift under 
the \emph{Frame-Based-Drift-Plus-Penalty-Algorithm} 
%\emph{FAB}
is given by:
\begin{align}
&\Delta(t_k) - V \expect{\sum_{t=t_k}^{t_{k+1}-1} R_{su}^{fab}(t) | \boldsymbol{Q}(t_k)} \leq  B + (Q_{su}(t_k) - V) \nonumber\\
&\times \expect{\sum_{t=t_k}^{t_{k+1}-1} R_{su}^{fab}(t) | \boldsymbol{Q}(t_k)} - X_{su}(t_k) \expect{T[k] P_{avg} | \boldsymbol{Q}(t_k)} \nonumber\\
&-  \expect{\sum_{t=t_k}^{t_{k+1}-1} (Q_{su}(t_k) \mu_{su}^{fab}(t) - X_{su}(t_k) P_{su}^{fab}(t)) | \boldsymbol{Q}(t_k)}
\label{eq:drift3}
\end{align}
Using Lemma \ref{lem:adm_ctrl_bound_lem}, we have that:
\begin{align*}
&\expect{\sum_{t=t_k}^{t_{k+1}-1} (Q_{su}(t_k) - V) R^{fab}_{su}(t) | \boldsymbol{Q}(t_k)} \leq \\
&C + \expect{\sum_{t=t_k}^{t_{k+1}-1} (Q_{su}(t_k) - V) R^{alt}_{su}(t) | \boldsymbol{Q}(t_k)}
 %\label{eq:adm_ctrl_bound1}
\end{align*}
Next, note that under the \emph{ALT} algorithm, we have:
\begin{align*}
   %\frac{(U_{su}(r) - V) \expect{\sum_{t=t[r]}^{t[r+1]-1} R_{su}(t) | \boldsymbol{Q}(r)}}{\expect{T[r] | \boldsymbol{Q}(r)}} \leq 
&\frac{\expect{\sum_{t=t_k}^{t_{k+1}-1} (Q_{su}(t_k) - V) R_{su}^{alt}(t) | \boldsymbol{Q}(t_k)}}{\expect{T[k] | \boldsymbol{Q}(t_k)}}  \\
&\leq \frac{\expect{\sum_{t=t_k}^{\hat{t}_{k+1}-1} (Q_{su}(t_k) - V) {R}_{su}^{stat}(t) | \boldsymbol{Q}(t_k)}}{\expect{\hat{T}[k] | \boldsymbol{Q}(t_k)}}
\end{align*}
 %where ${R}_{su}^{stat}(t), \hat{T}[k]$ denote the admission control decisions and frame length under \emph{STAT}. 
To see this, we have two cases:
\begin{enumerate}
\item $Q_{su}(t_k) > V$: Then, $R_{su}^{alt}(t) = 0$ for all $t \in \{t_k, t_k+1, \ldots, t_{k+1} -1 \}$, so that the left hand side above is $0$ 
while the right hand side is $\geq 0$. Hence, the inequality follows.

\item $Q_{su}(t_k) \leq V$: Then, $R_{su}^{alt}(t) = A_{su}(t)$ for all $t \in \{t_k, t_k+1, \ldots, t_{k+1} -1 \}$, so that the left hand side
becomes $(Q_{su}(t_k) - V) \lambda_{su}$ while the right hand side cannot be smaller than $(Q_{su}(t_k) - V) \lambda_{su}$. 

\end{enumerate}
Combining these, we get:
\begin{align*}
 %\frac{(U_{su}(r) - V) \expect{\sum_{t=t[r]}^{t[r+1]-1} R_{su}(t) | \boldsymbol{Q}(r)}}{\expect{T[r] | \boldsymbol{Q}(r)}} \leq 
&{(Q_{su}(t_k) - V) \expect{\sum_{t=t_k}^{t_{k+1}-1} R_{su}^{fab}(t) | \boldsymbol{Q}(t_k)}} \leq C \\ 
& + (Q_{su}(t_k) - V) \expect{\sum_{t=t_k}^{\hat{t}_{k+1}-1} {R}_{su}^{stat}(t) | \boldsymbol{Q}(t_k)} 
\frac{\expect{T[k] | \boldsymbol{Q}(t_k)}}{\expect{\hat{T}[k] | \boldsymbol{Q}(t_k)}}
 %\label{eq:ratio2}
\end{align*}

\begin{comment}
Before proceeding, we present a property of 
the  \emph{Frame-Based-Drift-Plus-Penalty-Algorithm} 
that will be useful in analyzing its performance later. 
Let $\mu_{su}^{fab}(t), P_{su}^{fab}(t)$ denote the resource allocation decisions
under this algorithm. 
From (\ref{eq:femto_resource_alloc}), it follows that:
\begin{align}
 %&\expect{U_{su}(k) \sum_{t=t[k]}^{t[k+1]-1} \mu_{su}^{fab}(t) - X_{su}(k) \sum_{t=t[k]}^{t[k+1]-1} P_{su}^{fab}(t) | \boldsymbol{Q}(k)} \nonumber\\
 %& \geq \expect{U_{su}(k) \sum_{t=t[k]}^{\hat{t}[k+1]-1} \hat{\mu}_{su}(t) - X_{su}(k) \sum_{t=t[k]}^{\hat{t}[k+1]-1} \hat{P}_{su}(t) | \boldsymbol{Q}(k)} \nonumber\\
 %& \times \frac{\expect{T[k] | \boldsymbol{Q}(k)}} {\expect{\hat{T}[k] | \boldsymbol{Q}(k)}}
&\expect{\sum_{t=t_k}^{t_{k+1}-1} \Big(U_{su}(t_k) \mu_{su}^{fab}(t) - X_{su}(t_k) P_{su}^{fab}(t) \Big) | \boldsymbol{Q}(t_k)} \geq\nonumber\\
& \expect{\sum_{t=t_k}^{\hat{t}_{k+1}-1} \Big(U_{su}(t_k) \hat{\mu}_{su}(t) - X_{su}(t_k) \hat{P}_{su}(t) \Big)| \boldsymbol{Q}(t_k)} \nonumber \\
& \times \frac{\expect{T[k] | \boldsymbol{Q}(t_k)}} {\expect{\hat{T}[k] | \boldsymbol{Q}(t_k)}}
\label{eq:prop1}
\end{align}
where $\hat{\mu}_{su}(t), \hat{P}_{su}(t)$ denote the resource allocation decisions
and $\hat{t}_{k+1}$ and $\hat{T}[k]$ denote the start of the $(k+1)^{th}$ frame and the length of
the $k^{th}$ frame, respectively, under any other control algorithm.
\end{comment}

Finally, since the resource allocation part of the \emph{Frame-Based-Drift-Plus-Penalty-Algorithm}
maximizes the ratio in (\ref{eq:femto_resource_alloc}), we have: 
\begin{align*}
 %&{\expect{U_{su}(k) \sum_{t=t[k]}^{t[k+1]-1} \mu_{su}^{fab}(t) - X_{su}(k) \sum_{t=t[k]}^{t[k+1]-1} P_{su}^{fab}(t) | \boldsymbol{Q}(k)}}
 %\geq \nonumber \\
 %& \expect{U_{su}(k) \sum_{t=t[k]}^{\hat{t}[k+1]-1} {\mu}_{su}^{stat}(t) - X_{su}(k) \sum_{t=t[k]}^{\hat{t}[k+1]-1} {P}_{su}^{stat}(t) | \boldsymbol{Q}(k)}
 %\frac{\expect{T[k] | \boldsymbol{Q}(k)}}{\expect{\hat{T}[k] | \boldsymbol{Q}(k)}}
&{\expect{\sum_{t=t_k}^{t_{k+1}-1} (Q_{su}(t_k)\mu_{su}^{fab}(t) - X_{su}(t_k) P_{su}^{fab}(t)) | \boldsymbol{Q}(t_k)}} \geq \\
& \expect{\sum_{t=t_k}^{\hat{t}_{k+1}-1} (Q_{su}(t_k){\mu}_{su}^{stat}(t) - X_{su}(t_k){P}_{su}^{stat}(t)) | \boldsymbol{Q}(t_k)} \\
&\times \frac{\expect{T[k] | \boldsymbol{Q}(t_k)}}{\expect{\hat{T}[k] | \boldsymbol{Q}(t_k)}}
 %\label{eq:ratio3}
\end{align*}
 %where ${P}_{su}^{fab}(t), {\mu}_{su}^{fab}(t)$ denote the resource allocation decisions under 
 %the \emph{Frame-Based-Drift-Plus-Penalty-Algorithm}
  %%\emph{FAB}
 %and ${P}_{su}^{stat}(t), {\mu}_{su}^{stat}(t)$ and $\hat{T}[k]$ denote the resource allocation decisions and
 %frame length, respectively, under \emph{STAT}.

Using these in (\ref{eq:drift3}), we have:
\begin{align*}
&\Delta(t_k) - V \expect{\sum_{t=t_k}^{t_{k+1}-1} R_{su}^{fab}(t) | \boldsymbol{Q}(t_k)} \leq  B + C \\ 
&+ (Q_{su}(t_k) - V) \expect{\sum_{t=t_k}^{\hat{t}_{k+1}-1} {R}_{su}^{stat}(t) | \boldsymbol{Q}(t_k)} 
\frac{\expect{T[k] | \boldsymbol{Q}(t_k)}}{\expect{\hat{T}[k] | \boldsymbol{Q}(t_k)}} \nonumber\\
& - \expect{\sum_{t=t_k}^{\hat{t}_{k+1}-1} (Q_{su}(t_k){\mu}_{su}^{stat}(t) - X_{su}(t_k) {P}_{su}^{stat}(t)) | \boldsymbol{Q}(t_k)} \\
& \times \frac{\expect{T[k] | \boldsymbol{Q}(t_k)}}{\expect{\hat{T}[k] | \boldsymbol{Q}(t_k)}} - X_{su}(t_k) \expect{T[k] P_{avg} | \boldsymbol{Q}(t_k)} 
  %&\qquad - Q_{su}(r) \expect{\sum_{t=t[r]}^{t[r+1]-1} \mu_{su}(t) | \boldsymbol{Q}(r)} + X_{su}(r) \expect{\sum_{t=t[r]}^{t[r+1]-1} C_{su}(t) | \boldsymbol{Q}(r)} 
 %\label{eq:drift4}
\end{align*}

Using (\ref{eq:iid_1})-(\ref{eq:iid_3}) in the inequality above, we get:
\begin{align}
\Delta(t_k) &- V \expect{\sum_{t=t_k}^{t_{k+1}-1} R_{su}^{fab}(t) | \boldsymbol{Q}(t_k)} \leq B + C \nonumber \\
&\qquad - V \upsilon^* \expect{T[k] | \boldsymbol{Q}(t_k)}
\label{eq:drift5}
\end{align}

To prove (\ref{eq:x_max1}), we rearrange (\ref{eq:drift5}) to get:
\begin{align*}
&\Delta(t_k) \leq B + C - V \upsilon^* \expect{T[k] | \boldsymbol{Q}(t_k)} \\& 
+ V \expect{\sum_{t=t_k}^{t_{k+1}-1} R_{su}^{fab}(t) | \boldsymbol{Q}(t_k)} \leq B + C + V T_{max} A_{max}
\end{align*}
(\ref{eq:x_max1}) now follows from Theorem 4.1 of \cite{neely_new}. Since
$X_{su}(t_k)$ is mean rate stable,  (\ref{eq:x_max2}) follows from Theorem 2.5(b) of \cite{neely_new}.
To prove (\ref{eq:femto_utility_bound}), we take expectations of both sides of (\ref{eq:drift5}) to get:
\begin{align*}
\expect{L(\boldsymbol{Q}(t_{k+1}))} &- \expect{L(\boldsymbol{Q}(t_k))} - V \expect{\sum_{t=t_k}^{t_{k+1}-1} R_{su}^{fab}(t)} \\
&\leq B + C - V \upsilon^* \expect{T[k]}
\end{align*}

Summing over $k \in \{1, 2, \ldots, K \}$, dividing by $V$, and rearranging yields:
\begin{align*}
&\sum_{k=1}^K \expect{\sum_{t=t_k}^{t_{k+1}-1} R_{su}^{fab}(t)}  \geq  
 %r^* \sum_{k=1}^K \expect{T[k]} - \frac{(B + C)K}{V} + \frac{\expect{L(\boldsymbol{Q}(K+1))} - \expect{L(\boldsymbol{Q}(1))}}{V} \geq 
\upsilon^* \sum_{k=1}^K \expect{T[k]} - \frac{(B + C)K}{V} 
\end{align*}
where we used that fact that $\expect{L(\boldsymbol{Q}(t_{K+1}))} \geq 0$ and $\expect{L(\boldsymbol{Q}(t_1))} = 0$. From this, we have:
\begin{align*}
\frac{\sum_{k=1}^K \expect{\sum_{t=t_k}^{t_{k+1}-1} R_{su}^{fab}(t)}}{\sum_{k=1}^K \expect{T[k]}}  
&\geq  \upsilon^* - \frac{(B + C)K}{V \sum_{k=1}^K \expect{T[k]}}  \\ 
&\geq \upsilon^* - \frac{B + C}{V T_{min}}
\end{align*}
since $\sum_{k=1}^K \expect{T[k]} \geq K T_{min}$. This proves (\ref{eq:femto_utility_bound}).

\begin{comment}

Under the \emph{Renewal-Based-Drift-Plus-Penalty-Algorithm}, we have the following:
\begin{align}
 %\frac{(U_{su}(r) - V) \expect{\sum_{t=t[r]}^{t[r+1]-1} R_{su}(t) | \boldsymbol{Q}(r)}}{\expect{T[r] | \boldsymbol{Q}(r)}} \leq 
{(U_{su}(k) - V) \expect{\sum_{t=t[k]}^{t[k+1]-1} R_{su}(t) | \boldsymbol{Q}(k)}} \leq 
(U_{su}(k) - V) 
\expect{\sum_{t=t[k]}^{\hat{t}[k+1]-1} \hat{R}_{su}(t) | \boldsymbol{Q}(k)}
\frac{\expect{T[k] | \boldsymbol{Q}(k)}}{\expect{\hat{T}[k] | \boldsymbol{Q}(k)}}
\label{eq:ratio2}
\end{align}
where $\hat{R}_{su}(t), \hat{T}[k]$ denote the admission control decisions and frame length under any other i.i.d. control algorithm.

\end{comment}

\section*{Appendix C \\ Computing D}

Here, we compute a finite $D$ that satisfies (\ref{eq:second_moment}). 
First, note that $\expect{T^2[k]}$ would be maximum
when the secondary user never cooperates. Next, let $I[k]$ and $B[k]$ denote the lengths of the primary user idle and busy periods,
respectively, in the $k^{th}$ frame. Thus, we have $T[k] = I[k] + B[k]$. 

In the following, we drop $[k]$ from the notation 
for convenience. Using the independence of $I$ and $B$, we have:
\begin{align*}
\expect{T^2} = \expect{I^2} + \expect{B^2} + 2 \expect{I} \expect{B}
\end{align*}
We note that $I$ is a geometric r.v. with parameter $\lambda_{pu}$. Thus, 
$\expect{I} = 1/\lambda_{pu}$ and $\expect{I^2} = (2 - \lambda_{pu})/\lambda_{pu}^2$.
To calculate $\expect{B}$, we apply Little's Theorem to get:
\begin{align*}
\expect{I} = \Big(1 - \frac{\lambda_{pu}}{\phi_{nc}} \Big) (\expect{I} + \expect{B})	
\end{align*}
This yields $\expect{B} = 1/(\phi_{nc} - \lambda_{pu})$. To calculate $\expect{B^2}$, we
use the observation that changing the service order of packets in the primary queue to preemptive LIFO 
does not change the length of the busy period $B$. However, with LIFO scheduling, $B$ now equals the duration that the first
packet stays in the queue. Next, suppose there are $N$ packets that interrupt the service of the first packet.
Let these be indexed as $\{1, 2, \ldots, N\}$.
We can relate $B$ to the service time $X$ of the first packet and the durations for which all these other packets 
stay in the queue as follows:
\begin{align}
B = X + \sum_{i=1}^N B_i
\label{eq:b_k}
\end{align}
Here, $B_i$ denotes the duration for which packet $i$ stays in the queue. 
Using the memoryless property of the i.i.d. arrival process of the primary packets as well as the
i.i.d. nature of the service times, it follows that all the r.v.'s $B_i$ are i.i.d. with the \emph{same}
distribution as $B$. Further, they are independent of $N$.
Squaring (\ref{eq:b_k}) and taking expectations, we get:
\begin{align}
\expect{B^2} &= \expect{X^2} + 2 \expect{X} \expect{N} \expect{B} \nonumber\\&
+ \expect{\Big(\sum_{i=1}^N B_i\Big)^2} 
\label{eq:b_squared}
\end{align}
Note that $X$ is a geometric r.v. with parameter $\phi_{nc}$. Thus $\expect{X} = 1/\phi_{nc}$ and 
$\expect{X^2} = (2 - \phi_{nc})/\phi_{nc}^2$. Also, $\expect{N} = \lambda_{pu} \expect{X} = \lambda_{pu}/\phi_{nc}$.
Using these in (\ref{eq:b_squared}), we have:
\begin{align*}
\expect{B^2} = \frac{(2 - \phi_{nc})}{\phi_{nc}^2} + \frac{2 \lambda_{pu}}{\phi_{nc}^2 (\phi_{nc} - \lambda_{pu})} + \expect{\Big(\sum_{i=1}^N B_i\Big)^2} 
 %\label{eq:b_squared1}
\end{align*}
To calculate the last term, we have:
\begin{align*}
&\expect{\Big(\sum_{i=1}^N B_i\Big)^2} = \expect{\sum_{i=1}^N B_i^2} + 2 \expect{\sum_{i \neq j} B_i B_j} \\
&\qquad = \expect{N} \expect{B^2} + 2 (\expect{B})^2(\expect{N^2} -  \expect{N})
\end{align*}
Note that given $X=x$, $N$ is a binomial r.v. with parameters $(x, \lambda_{pu})$. Thus, we have:
\begin{align*}
\expect{N^2} &= \sum_{x \geq 1} \expect{N^2 | X=x} \textrm{Prob}[X=x] \\
&= \sum_{x \geq 1} \Big[(x \lambda_{pu})^2 + x \lambda_{pu}(1 - \lambda_{pu})\Big] (1-\phi_{nc})^{x-1}\phi_{nc} \\
&= \lambda_{pu}^2  \sum_{x \geq 1} x^2 \phi_{nc} (1 -\phi_{nc})^{x-1} \\
&+ \lambda_{pu} (1 - \lambda_{pu}) \sum_{x \geq 1} x \phi_{nc} (1 -\phi_{nc})^{x-1} \\
&= \lambda_{pu}^2 \frac{(2 - \phi_{nc})}{\phi_{nc}^2} + \lambda_{pu} (1 - \lambda_{pu}) \frac{1}{\phi_{nc}}
\end{align*}
Using this, we have:
\begin{align*}
&\expect{\Big(\sum_{i=1}^N B_i\Big)^2} \\
& = \frac{\lambda_{pu}}{\phi_{nc}} \expect{B^2} + 2  \Big(\frac{1}{\phi_{nc} - \lambda_{pu}}\Big)^2 (\expect{N^2} -  \expect{N}) \\
&=  \frac{\lambda_{pu}}{\phi_{nc}} \expect{B^2} + 2  \Big(\frac{1}{\phi_{nc} - \lambda_{pu}}\Big)^2 \Big(\frac{2 \lambda_{pu}^2 (1 - \phi_{nc})}{\phi_{nc}^2}\Big)
\end{align*}
Using this, we have:
\begin{align*}
\expect{B^2} &= \frac{(2 - \phi_{nc})}{\phi_{nc}^2} + \frac{2 \lambda_{pu}}{\phi_{nc}^2 (\phi_{nc} - \lambda_{pu})} +
\frac{\lambda_{pu}}{\phi_{nc}} \expect{B^2} \\
&+ 2  \Big(\frac{1}{\phi_{nc} - \lambda_{pu}}\Big)^2 \Big(\frac{2 \lambda_{pu}^2 (1 - \phi_{nc})}{\phi_{nc}^2}\Big)
\end{align*}
Simplifying this yields:
\begin{align*}
&\expect{B^2} \nonumber\\
&= \frac{(2 - \phi_{nc})}{\phi_{nc} (\phi_{nc} - \lambda_{pu})} + \frac{2 \lambda_{pu}}{\phi_{nc} (\phi_{nc} - \lambda_{pu})^2}
+ \frac{4 \lambda_{pu}^2 (1 - \phi_{nc})} {\phi_{nc} (\phi_{nc} - \lambda_{pu})^3}
%\label{eq:b_squared2}
\end{align*}

\end{document}